\documentclass[12pt,reqno]{amsart}

\usepackage[utf8]{inputenc}
\usepackage{amsmath, amsfonts, amsthm, amssymb}
\usepackage[margin=1in]{geometry}

\usepackage{bxeepic}
\usepackage{pstricks,pst-all,pst-pdf}
\usepackage{pifont}
\usepackage{hhline}
\usepackage{eepic}

\usepackage{mathtools}
\usepackage{float}
\usepackage{hyperref}
\usepackage{wrapfig}
\usepackage{array}

\usepackage{perpage}
\usepackage{tcolorbox}
\usepackage{fdsymbol}

\newcommand{\tsc}{\fontfamily{cmr}\small\selectfont\scshape}
\newcommand{\liber}{\fontfamily{LinuxLibertineT-TLF}\selectfont}

\newcommand\settxtfont{\liber}

\newcommand\setcapfont{\tsc}

\allowdisplaybreaks[4]

\newlength\gvhs
\newlength\gvls
\newlength\gvps
\newlength\gvms
\newlength\gvss
\newlength\gvts
\newlength\gvxs
\newcommand{\vlskip}{\vspace{\gvls}}

\newcommand{\vmskip}{\vspace{\gvms}}
\newcommand{\vsskip}{\vspace{\gvss}}
\newcommand{\vtskip}{\vspace{\gvts}}

\MakePerPage{footnote}
\newcommand\ftnnote[1]{\footnote{\ #1}}

\newcommand\boxftnnote[1]{\,*\footnotetext{* #1}}

\newcommand\hlabel{\phantomsection\label}

\newenvironment{mcenter}%
{\parskip=0pt\par\nopagebreak\centering}%
{\par\noindent\ignorespacesafterend}%

\newcommand\ifempty[1]{\def\temp{#1}\ifx\temp\empty}

\newcommand\BeginPic{\begingroup}
\newcommand\EndPic[1][0.1pt]{\vspace{#1}\endgroup\WFclear}

\makeatletter

\newcounter{wt@lines}

\define@key{wtkeys}{width}{\def\wt@width{#1}\relax}
\define@key{wtkeys}{w}{\def\wt@width{#1}\relax}
\define@key{wtkeys}{scale}{\def\wt@scale{#1}\relax}
\define@key{wtkeys}{s}{\def\wt@scale{#1}\relax}
\define@key{wtkeys}{lines}{\setcounter{wt@lines}{#1}\relax}
\define@key{wtkeys}{l}{\setcounter{wt@lines}{#1}\relax}
\define@key{wtkeys}{vskip}{\def\wt@vskip{#1}\relax}
\define@key{wtkeys}{v}{\def\wt@vskip{#1}\relax}
\define@key{wtkeys}{pos}{\def\wt@side{#1}\relax}
\define@key{wtkeys}{p}{\def\wt@side{#1}\relax}
\define@key{wtkeys}{caption}{\def\wt@caption{#1}\relax}
\define@key{wtkeys}{c}{\def\wt@caption{#1}\relax}

\newcommand{\wtext}[2]%
{%
  \def\wt@width{1.0}%
  \def\wt@scale{1.0}%
  \setcounter{wt@lines}{-1}%
  \def\wt@vskip{-2ex}%
  \def\wt@side{r}%
  \let\wt@caption\empty%
  \setkeys{wtkeys}{#2}%
  \ifx\empty\wt@caption\relax\def\cptext{\empty}\else
                             \def\cptext{\par{\setcapfont \wt@caption}\vspace{-5pt}}\fi
  \ifnum\thewt@lines<0\relax
      \begin{wrapfigure}{\wt@side}{\wt@width\textwidth}%
      {\vspace{\wt@vskip}{#1}}%
      \cptext
      \end{wrapfigure}%
  \else
      \begin{wrapfigure}[\thewt@lines]{\wt@side}{\wt@width\textwidth}%
      {\vspace{\wt@vskip}{#1}}%
      \cptext
      \end{wrapfigure}%
  \fi
}

\makeatother

\newcommand\fxbox[2][3pt]{\setlength{\fboxrule}{#1}\framebox[1.1\width]{\hspace{2pt}#2\hspace{2pt}}}

\newcommand{\gte}{\geqslant}
\newcommand{\lte}{\leqslant}

\newtheorem{theorem}{Theorem}[section]
\newtheorem{lemma}[theorem]{Lemma}
\newtheorem{corollary}[theorem]{Corollary}

\newtheorem{question}[theorem]{Question}

\newcommand\note{\vsskip\noindent\textbf{Remark}}

\newenvironment{mproof}{\vtskip\noindent\textbf{\underline{Proof}}. \ignorespaces}{\ignorespaces\kern0.1em~\ensuremath{\square}\smallbreak}
\newenvironment{mproof*}{\vtskip\textbf{\underline{Proof}}. \ignorespaces}{\smallbreak}

\newcommand{\mmod}[1]{\ (\mathrm{mod}\ #1)}

\newcommand\isdiv{\mathrel{\smash{\lower.1ex\hbox{$\vdots$}}}}
\newcommand\notdiv{\kern-.03ex\mathrel{\hbox{$\not{\kern.03ex\isdiv}$}}}

\newcommand\pp{\,.}
\newcommand\pc{\,,}

\newcommand\veps{\varepsilon}

\newcommand\Z{\mathbb Z}

\newcommand{\mcb}{\mathcal{B}}
\newcommand{\mcc}{\mathcal{C}}
\newcommand{\mcl}{\mathcal{L}}
\newcommand{\mcz}{\mathcal{Z}}

\newcommand{\mdph}[4]{$\mathcal{M\kern0.15ex}^{#4}_{#1,#2,#3}$}
\newcommand{\ndph}[3]{$\mathcal{N\kern0.25ex}^{#3}_{#1,#2}$}
\newcommand{\pdph}[3]{$\mathcal{P\kern0.25ex}^{#3}_{#1,#2}$}
\newcommand{\wstar}{\largewhitestar}

\newtagform{brackets}[\textbf]{$\langle$}{$\rangle$}
\usetagform{brackets}

\newcommand\eqtlabel[1]{$\langle #1\rangle$}

\newcommand{\xxx}{\medskip\begin{center}*\quad *\quad *\end{center}\medskip}

\tcbuselibrary{skins}

\newtcolorbox{cmbox}{colback=blue!20!white, width=\linewidth-12mm,
before={\vsskip\begin{mcenter}}, after={\end{mcenter}}, halign=left, before upper={\parindent3ex}, arc=3mm, boxsep=2mm, left=1mm, right=1mm}

\newtcolorbox{thbox}{colback=black!10!white, width=\linewidth-30mm, before={\vsskip\begin{mcenter}}, after={\end{mcenter}}, halign=left, before upper={\noindent\kern0.01pt\parindent=3ex},, arc=3mm, left=5pt, right=5pt}

\newcounter{qnindex}
\newtcolorbox{qnxbox}[1]{center, title={\refstepcounter{qnindex}$\mathbf{\thesection.\theqnindex.}$\quad#1}, colback=red!10!white, coltitle = black, colbacktitle=black!20!white, width=\linewidth, halign=left, before upper={\parindent3ex}, before={\vsskip\begin{mcenter}}, after={\end{mcenter}}, arc=1mm, boxsep=2mm, left=1mm, right=1mm}

\begin{document}

\sloppy
\allowdisplaybreaks
\interfootnotelinepenalty=10000
\hbadness = 10000
\vbadness = 10000


\settxtfont


\newcommand\nck[2]{\ensuremath{\langle #1 \,\text{$\times$\,\llap{$\times$}}\, #2\rangle}}
\newcommand\kcb[1]{\ensuremath{\text{K}_{#1}}}
\newcommand\fxkcb[1]{\noindent\fxbox[2pt]{\,\kcb{#1}\,}}


\title{Closed polylines with fixed self-intersection index}
\author{Dmitri V.~Fomin}
\address{Boston, USA}
\email{fomin@hotmail.com}
\date{\today}
\begin{abstract}
We investigate the existence of closed polylines (also known as closed polygonal chains or self-crossing polygons) that intersect each of their edges the same number of times. The most general question in this corner of combinatorial geometry asks for all pairs $(n, k)$ such that there exists a closed polyline with $n$ edges, each intersecting the same polyline exactly $k$ times. For $k = 1$ and $k = 2$, this is a very simple question answered several decades ago. In this article, we present a complete solution for $k = 3, 4, 6$ and $n=42$, as well as the proof of some non-existence theorems. In conclusion, we show that, for an arbitrary positive integer $k$, a polyline of the required type exists for any sufficiently large integer $n$ such that $nk$ is even.
\end{abstract}

\maketitle


\vsskip

\centerline{%
\setlength{\unitlength}{9mm}
\begin{picture}(9,8)(0,-1)
  \color{red}
  \linethickness{0.33mm}
  \path(0.00, 1.00)(4.00, 4.00)
  \path(4.00, 4.00)(5.00, 0.00)
  \path(5.00, 0.00)(5.00, 5.00)
  \path(5.00, 5.00)(4.00, 0.00)
  \path(4.00, 0.00)(7.00, 5.00)
  \path(7.00, 5.00)(2.00, 7.00)
  \path(2.00, 7.00)(4.00, 1.00)
  \path(4.00, 1.00)(8.00, 6.00)
  \path(8.00, 6.00)(4.00, 3.00)
  \path(4.00, 3.00)(3.00, 7.00)
  \path(3.00, 7.00)(3.00, 2.00)
  \path(3.00, 2.00)(4.00, 7.00)
  \path(4.00, 7.00)(1.00, 2.00)
  \path(1.00, 2.00)(6.00, 0.00)
  \path(6.00, 0.00)(4.00, 6.00)
  \path(4.00, 6.00)(0.00, 1.00)
  \color{black}
  \put(0.00, 1.00){\circle*{0.17}}
  \put(4.00, 4.00){\circle*{0.17}}
  \put(5.00, 0.00){\circle*{0.17}}
  \put(5.00, 5.00){\circle*{0.17}}
  \put(4.00, 0.00){\circle*{0.17}}
  \put(7.00, 5.00){\circle*{0.17}}
  \put(2.00, 7.00){\circle*{0.17}}
  \put(4.00, 1.00){\circle*{0.17}}
  \put(8.00, 6.00){\circle*{0.17}}
  \put(4.00, 3.00){\circle*{0.17}}
  \put(3.00, 7.00){\circle*{0.17}}
  \put(3.00, 2.00){\circle*{0.17}}
  \put(4.00, 7.00){\circle*{0.17}}
  \put(1.00, 2.00){\circle*{0.17}}
  \put(6.00, 0.00){\circle*{0.17}}
  \put(4.00, 6.00){\circle*{0.17}}
  \put(2,-1){$n = 16 \pc\ k = 5$}
\end{picture}
\kern6ex
\begin{picture}(7,8)(0,-1)
	\linethickness{0.33mm}
	\color{red}
	\path(1.50,0.00)(5.50,3.25)(4.50,7.00)(3.50,3.00)(0.00,4.66)
	\path(0.00,4.66)(3.50,2.00)(5.50,7.00)(1.50,3.75)(2.50,0.00)(3.50,4.00)(7.00,2.33)
	\path(7.00,2.33)(3.50,5.00)(1.50,0.00)
	\color{black}
	\put(1.50,0.00){\circle*{0.17}}
	\put(5.50,3.25){\circle*{0.17}}
	\put(4.50,7.00){\circle*{0.17}}
	\put(3.50,3.00){\circle*{0.17}}
	\put(0.00,4.66){\circle*{0.17}}
	\put(3.50,2.00){\circle*{0.17}}
	\put(5.50,7.00){\circle*{0.17}}
	\put(1.50,3.75){\circle*{0.17}}
	\put(2.50,0.00){\circle*{0.17}}
	\put(3.50,4.00){\circle*{0.17}}
	\put(7.00,2.33){\circle*{0.17}}
	\put(3.50,5.00){\circle*{0.17}}
	\put(2,-1){$n=12\pc\ k=3$}
\end{picture}%
}

\vmskip

\section{Introduction}

Let us say that a flat closed polyline has type \nck{n}{k} (or satisfies condition \nck{n}{k}) if the polyline consists of $n$ edges, and each edge is intersected exactly $k$ times by the rest of the polyline. The same definition can be extended to any collection of segments on the plane.

It is assumed that all these intersections are transversal (i.e., no two intersecting edges are parallel) and occur only in the interior of the edges; a vertex of the polyline is not considered a point of self-intersection. Since we are only interested in the number of self-intersections, we can always perform a tiny ``perturbation'' of the polyline's vertices to ensure that all the self-intersection points are pairwise distinct. The figure above shows two polylines of types \nck{16}{5} and~\nck{12}{3}, respectively---for instance, the left one is a closed polygonal chain with sixteen edges such that each of them intersects exactly five other edges of this chain.

In this case we will call that number $k$ the \textit{self-intersection index} of the given polyline. We will also say that if a \nck{n}{k} polyline exists, then type \nck{n}{k} is \textit{feasible}.

Quite naturally, the following general question arises:

\vsskip

\begin{cmbox}
For which values of $n$ and $k$ does polyline of type \nck nk exist?\boxftnnote{As a particular case, consider this question: for which natural $k$ does there exist a type~\nck{42}{k} polyline? This exploratory problem was offered more than $50$ years ago at the district round of the Leningrad City Mathematical Olympiad. In section \ref{sec:b42} we provide a more or less complete solution to that question.}
\end{cmbox}

\vsskip

In this article we will investigate numerous cases of this general problem, providing complete answers for cases $k=1$, $2$, $3$, $4$ and $6$, as well as for~$n=42$. For $k=5$, only type \nck{12}{5} remains a mystery. We will also prove the theorem that shows that for any fixed self-intersection index $k$ a polyline of type \nck nk exists for all sufficiently large values of $n$ provided that $nk$ is even. Namely, type \nck nk is feasible if either $k$ is even and $n \gte 2k+3$, or if $k$ is odd and $n$ is an even number such that~$n\gte 8k+6$.

In the end we are posing several interesting unsolved problems related to this topic.

\section{Non-existence theorems}
\hlabel{sec:nonexist}

Let us start our investigation by introducing some useful notation. Namely, $\mcb_n$ will denote the set of all natural numbers $k$ for which a polyline of type \nck{n}{k} exists. Similarly, $\mcc_k$ denotes the set of all natural numbers $n$ for which it is possible to construct a polyline of type~\nck{n}{k}.

\vsskip

It is easy to see that the parity of numbers $n$ and $k$ is very important.

\begin{lemma}
\hlabel{thm:nk_must_be_even}
If both numbers $n$ and $k$ are odd, there is no polyline of type~\nck{n}{k}.
\end{lemma}

\begin{mproof}
Assuming the opposite, consider graph $G$ whose vertices are $n$ segments of the given polyline $L$, while the edges of the graph connect vertices (segments) which intersect each other. Since a graph cannot have an odd number of odd vertices, the obtained contradiction proves the lemma.
\end{mproof}

It is also obvious that polylines of type \nck{n}{k} do not exist if $k\gte n-2$. Indeed, an edge of a polyline cannot intersect itself as well as the two adjacent edges. In other words,
$$
nk \notdiv 2 \,\lor\, n \lte k+2 \implies k \notin \mcb_n \pc\ n \notin \mcc_k \pp
$$

\vtskip

Here is another, relatively simple, but not so well known fact.

\begin{lemma}
\hlabel{thm:nnminus3}
A polyline of type \nck{n}{n-3} exists if and only if $n \gte 5$ and $n$ is odd.
\end{lemma}

\begin{mproof}
For any odd $n \gte 5$ an obvious example of such a polyline is the one formed by all main diagonals of a regular polygon with $n$ sides.

Let us assume that such polyline $L$ exists for some even~$n$. Clearly, that is possible only if each of its edges intersects all other edges except for itself and the two adjacent edges. Consider an arbitrary edge $AB$ and two adjacent edges. These edges must intersect and, therefore, their other ends (different from points $A$ and $B$) lie to the same side of line~$AB$. Hence, if we remove the three above mentioned edges (including $AB$), then each of the remaining $n-3$ must intersect line~$AB$. However, the polyline formed by these edges connects two points that lie to the same side of $AB$ which is impossible since $n-3$ is odd.
\end{mproof}

\vtskip

Here is another non-existence theorem for one very concrete case~\nck{6}{2}.

\begin{lemma}
\hlabel{thm:nk62}
Polylines of type \nck{6}{2} do not exist.
\end{lemma}

\begin{mproof}
Assume that such polyline $L$ exists. Then each of its edges intersects all other edges except for itself, two neighbors and exactly one other edge. Index these edges in the order of the polyline traversal by the numbers $0$ to $5$, considering these indices as residues modulo~$6$. Then the edges can be denoted $e_k$ for~$k \in \Z/6\Z$, so that for every residue $k \mmod 6$ there is residue $k+d_k$ with $d_k=2$, $3$, or $4$ such that edges $e_k$ and $e_{k+d_k}$ do not intersect.

Now consider graph $G^*$ with six vertices (residues modulo~$6$), in which every vertex (polyline edge) $e_k$ is connected with exactly one other vertex---namely, with~$e_{k+d_k}$. These edges form a perfect matching in complete graph~$K_6$.

It is easy to see now that graph $G^*$ has one of the two following forms (there are only two essentially different options):

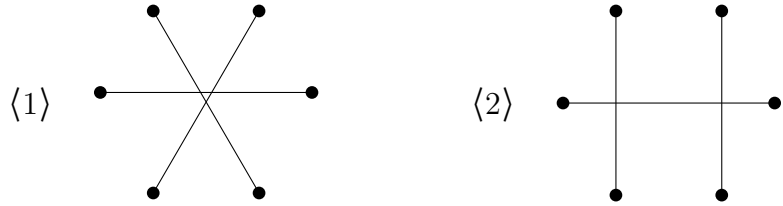
\vmskip

\centerline{%
\setlength{\unitlength}{2.8mm}

\begin{picture}(2,10)
	\put(0.00, 4.50){\eqtlabel{1}}
\end{picture}
\kern3ex
\begin{picture}(12,10)
	\put(10.00, 5.50){\circle*{0.6}}
	\put(7.50, 9.35){\circle*{0.6}}
	\put(2.50, 9.35){\circle*{0.6}}
	\put(0.00, 5.50){\circle*{0.6}}
	\put(2.50, 0.75){\circle*{0.6}}
	\put(7.50, 0.75){\circle*{0.6}}
	\path(10.00, 5.50)(0.00, 5.50)
	\path(7.50, 9.35)(2.50, 0.75)
	\path(2.50, 9.35)(7.50, 0.75)
\end{picture}
\kern8ex
\begin{picture}(2,10)
	\put(0.00, 4.50){\eqtlabel{2}}
\end{picture}
\kern3ex
\begin{picture}(12,10)
	\put(10.00, 5.00){\circle*{0.6}}
	\put(7.50, 9.35){\circle*{0.6}}
	\put(2.50, 9.35){\circle*{0.6}}
	\put(0.00, 5.00){\circle*{0.6}}
	\put(2.50, 0.65){\circle*{0.6}}
	\put(7.50, 0.65){\circle*{0.6}}
	\path(10.00, 5.00)(0.00, 5.00)
	\path(7.50, 9.35)(7.50, 0.65)
	\path(2.50, 9.35)(2.50, 0.65)
\end{picture}
}

\vmskip

In these drawings the vertices are arranged along the circle in the natural order of indexing modulo~$6$. Then graph $G$, showing edge intersections (see lemma~$1$), must look like this:

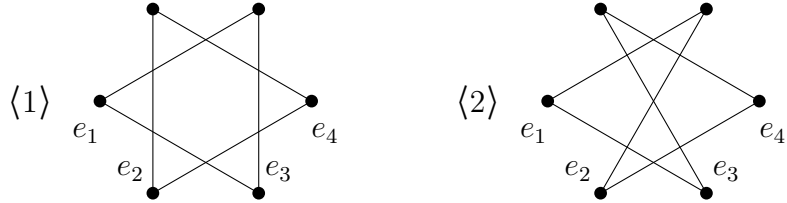
\vmskip

\centerline{%
\setlength{\unitlength}{2.8mm}
\begin{picture}(2,10)
	\put(0.00, 4.50){\eqtlabel{1}}
\end{picture}
\kern3ex
\begin{picture}(10,10)
	\put(10.00, 5.00){\circle*{0.6}}
	\put(7.50, 9.35){\circle*{0.6}}
	\put(2.50, 9.35){\circle*{0.6}}
	\put(0.00, 5.00){\circle*{0.6}}
	\put(2.50, 0.65){\circle*{0.6}}
	\put(7.50, 0.65){\circle*{0.6}}
	\path(0.00, 5.00)(7.50, 0.65)
	\path(0.00, 5.00)(7.50, 9.35)
	\path(10.00, 5.00)(2.50, 0.65)
	\path(10.00, 5.00)(2.50, 9.35)
	\path(7.50, 9.35)(7.50, 0.65)
	\path(2.50, 9.35)(2.50, 0.65)
	\put(-1.30, 3.30){$e_1$}
	\put(0.80, 1.50){$e_2$}
	\put(7.80, 1.50){$e_3$}
	\put(10.00, 3.30){$e_4$}
\end{picture}
\kern10ex
\begin{picture}(2,10)
	\put(0.00, 4.50){\eqtlabel{2}}
\end{picture}
\kern3ex
\begin{picture}(10,10)
	\put(10.00, 5.00){\circle*{0.6}}
	\put(7.50, 9.35){\circle*{0.6}}
	\put(2.50, 9.35){\circle*{0.6}}
	\put(0.00, 5.00){\circle*{0.6}}
	\put(2.50, 0.65){\circle*{0.6}}
	\put(7.50, 0.65){\circle*{0.6}}
	\path(0.00, 5.00)(7.50, 0.65)
	\path(0.00, 5.00)(7.50, 9.35)
	\path(10.00, 5.00)(2.50, 0.65)
	\path(10.00, 5.00)(2.50, 9.35)
	\path(7.50, 9.35)(2.50, 0.65)
	\path(2.50, 9.35)(7.50, 0.65)
	\put(-1.30, 3.30){$e_1$}
	\put(0.80, 1.50){$e_2$}
	\put(7.80, 1.50){$e_3$}
	\put(10.00, 3.30){$e_4$}
\end{picture}
}

\vmskip

In each of these two cases, the subgraph formed by the four bottom vertices of graph $G$ is the same. We can denote these four vertices (edges) by $e_1$, $e_2$, $e_3$, and $e_4$---then in  this $4$-tuple of edges only edges in pairs $(e_1, e_3)$ and~$(e_2, e_4)$ intersect. These four edges form a contiguous group, and therefore, the triple $e_1$, $e_2$, $e_3$ must form a configuration shown in the left figure below.

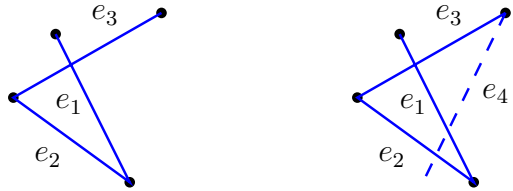
\vmskip

\centerline{%
\setlength{\unitlength}{2.8mm}
\begin{picture}(12,10)
	\put(4.00, 8.00){\circle*{0.5}}
	\put(7.50, 1.00){\circle*{0.5}}
	\put(2.00, 5.00){\circle*{0.5}}
	\put(9.00, 9.00){\circle*{0.5}}
	\put(4.00, 4.50){$e_1$}
	\put(3.00, 2.00){$e_2$}
	\put(5.70, 8.70){$e_3$}
	\color{blue}
	\linethickness{0.33mm}
	\path(4.00, 8.00)(7.50, 1.00)(2.00, 5.00)(9.00, 9.00)
\end{picture}
\kern6ex
\begin{picture}(12,10)
	\put(4.00, 8.00){\circle*{0.5}}
	\put(7.50, 1.00){\circle*{0.5}}
	\put(2.00, 5.00){\circle*{0.5}}
	\put(9.00, 9.00){\circle*{0.5}}
	\put(4.00, 4.50){$e_1$}
	\put(3.00, 2.00){$e_2$}
	\put(5.70, 8.70){$e_3$}
	\put(7.90, 5.00){$e_4$}
	\color{blue}
	\linethickness{0.33mm}
	\path(4.00, 8.00)(7.50, 1.00)(2.00, 5.00)(9.00, 9.00)
	\dashline{0.66}[0.66](9.00,9.00)(5.00,0.80)	
\end{picture}
}

But in this configuration it is clearly impossible to draw edge $e_4$ which intersects $e_2$ but not~$e_1$ (see the right figure above). This yields the contradiction, and, therefore, both cases \eqtlabel{1} and \eqtlabel{2} are impossible.
\end{mproof}

The statement of the last lemma can be generalized to the following neat proposition.

\begin{lemma}
\hlabel{thm:nnminus4}
A polyline of type \nck{n}{n-4} exists if and only if $n$ is divisible by~$4$.\ftnnote{Naturally, it is assumed that both parameters $n$ and $n-4$ are positive.}
\end{lemma}

\begin{mproof}
We begin with constructing an example of such polyline for $n = 4m > 4$. Then $n-4 = 4m-4 = 2(2m-2)$. Consider a regular polygon $M$ with $4m$ vertices and draw all diagonals of ``span'' equal to $2m-1$---in other words, diagonals that connect two vertices separated by a chain of exactly $2m-1$ consecutive sides of~$M$ (the sides themselves are not drawn). Since the number of vertices in $M$ is co-prime with this fixed span, it follows that these diagonals form one closed polyline. Observe that each drawn diagonal intersects exactly $2(2m-2)$ other such diagonals.

It remains to prove that for all other values of $n$ a polyline of the desired type does not exist. For odd $n$, both parameters $n$ and $n-4$ are odd, and from lemma $\ref{thm:nk_must_be_even}$ such a polyline cannot exist. Therefore, the only remaining case is $n = 4m+2$.

Assuming that such polyline $\mcl$ exists, index the edges of $\mcl$ by numbers $1$ through $4m+2$ in the order they are connected inside the polyline. Since every edge in $\mcl$ intersects all other edges except itself, its two neighbors, and exactly one other edge, it follows that the edges are split into pairs of non-adjacent non-intersecting segments. Consider an arbitrary pair $e_i$ and $e_j$ of such edges (we may assume $i < j$). Since they do not intersect, there exists a straight line $\ell$ that separates them. 

All other edges in $\mcl$ must intersect line~$\ell$. Indeed, let $e$ be some edge of our polyline, different from $e_i$ and~$e_j$. If $e$ is not adjacent to edge $e_i$, then $e$ has to intersect it. And if $e$ is adjacent to~$e_i$, then they obviously have a common point. Same is true for $e$ and $e_j$, and therefore, in either case edge $e$ has a common point with both $e_i$ and~$e_j$. Hence, it must intersect line~$\ell$.

This means that the chain of edges $e_{i}$, $e_{i+1}$, ..., $e_{j-1}$, $e_{j}$, which connects $e_i$ and $e_{j}$, intersects line $\ell$ exactly $j-i-1$ times. That number has to be odd, since $e_i$ and $e_j$ lie to the opposite sides of $\ell$---it follows that indices $i$ and $j$ have the same parity. Therefore, for every pair of non-adjacent non-intersecting edges the sum of their indices is even---hence, the sum of all indices is even. That sum, however, equals
$$
1 + 2 + ... + (4m+1) + (4m+2) = \frac 12 (4m+2)(4m+3) = (2m+1)(4m+3) \pc
$$
which is an odd number. This contradiction proves the lemma.
\end{mproof}

\section{\texorpdfstring{Computing $\mcc_k$: simple cases $k=1$ and $k=2$}{Simple cases: computing C-1 and C-2}}
\hlabel{sec:c12}

If the self-intersection index equals $1$, then the number of polyline's edges must be even since they are split into pairs of intersecting segments. Clearly, the number of edges cannot be $2$ or~$4$. For any larger even number of edges, such a polyline can be constructed. Below we show examples of these polylines for $n=6$, $n=8$, and~$n=10$.

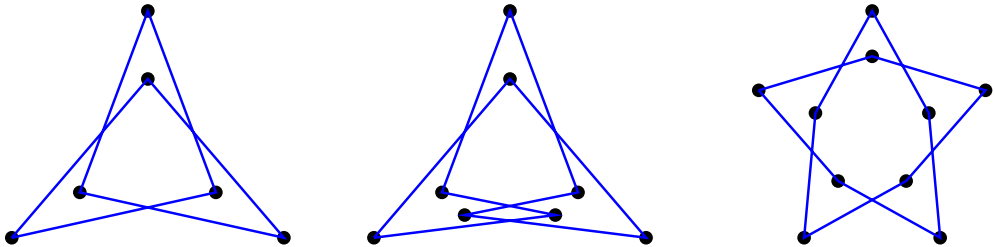
\vlskip

\centerline{%
\setlength{\unitlength}{3mm}
\begin{picture}(12,10)
	\put(0,0){\circle*{0.6}}
	\put(9,2){\circle*{0.6}}
	\put(6,10){\circle*{0.6}}
	\put(3,2){\circle*{0.6}}
	\put(12,0){\circle*{0.6}}
	\put(6,7){\circle*{0.6}}
	\color{blue}
	\linethickness{0.33mm}
	\path(0,0)(9,2)(6,10)(3,2)(12,0)(6,7)(0,0)
\end{picture}
\kern6ex
\begin{picture}(12,10)
	\put(0,0){\circle*{0.6}}
	\put(8,1){\circle*{0.6}}
	\put(3,2){\circle*{0.6}}
	\put(6,10){\circle*{0.6}}
	\put(9,2){\circle*{0.6}}
	\put(4,1){\circle*{0.6}}
	\put(12,0){\circle*{0.6}}
	\put(6,7){\circle*{0.6}}
	\color{blue}
	\linethickness{0.33mm}
	\path(0,0)(8,1)(3,2)(6,10)(9,2)(4,1)(12,0)(6,7)(0,0)
\end{picture}
\kern6ex
\begin{picture}(12,10)
	\put(1,6.5){\circle*{0.6}}
	\put(4.5,2.5){\circle*{0.6}}
	\put(9,0){\circle*{0.6}}
	\put(8.5,5.5){\circle*{0.6}}
	\put(6,10){\circle*{0.6}}
	\put(3.5,5.5){\circle*{0.6}}
	\put(3,0){\circle*{0.6}}
	\put(7.5,2.5){\circle*{0.6}}
	\put(11,6.5){\circle*{0.6}}
	\put(6,8){\circle*{0.6}}
	\color{blue}
	\linethickness{0.33mm}
	\path(1,6.5)(4.5,2.5)(9,0)(8.5,5.5)(6,10)(3.5,5.5)(3,0)(7.5,2.5)(11,6.5)(6,8)(1,6.5)
\end{picture}
}

\vmskip

Now we can construct the required polyline of any even size $n \gte 12$ by ``gluing'' two smaller polylines with the same self-intersection index. To do that, take polyline $L_1$ with $6$ edges and polyline $L_2$ with $n-6$ edges and position them on the plane so that they do not intersect each other except in one common vertex. Double that vertex as shown in the figure. In this way, we have merged these two polylines together into one polyline with $n$ edges.

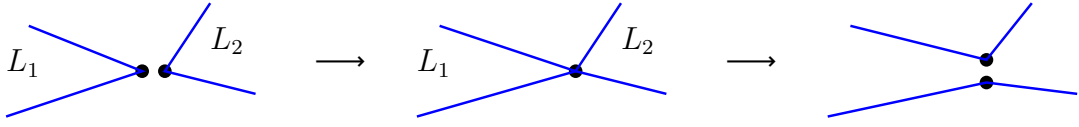
\vmskip

\centerline{
\setlength{\unitlength}{3mm}
\begin{picture}(12,5)
	\put(6,2){\circle*{0.6}}
	\put(7,2){\circle*{0.6}}
	\put(0,2){$L_1$}
	\put(9,3){$L_2$}
	\color{blue}
	\linethickness{0.33mm}
	\path(0,0)(6,2)(1,4)
	\path(11,1)(7,2)(9,5)
\end{picture}
\kern2ex
\begin{picture}(3,5)
	\put(0,2){$\longrightarrow$}
\end{picture}
\kern2ex
\begin{picture}(12,5)
	\put(7,2){\circle*{0.6}}
	\put(0,2){$L_1$}
	\put(9,3){$L_2$}
	\color{blue}
	\linethickness{0.33mm}
	\path(0,0)(7,2)(11,1)
	\path(1,4)(7,2)(9,5)
\end{picture}
\kern2ex
\begin{picture}(3,5)
	\put(0,2){$\longrightarrow$}
\end{picture}
\kern2ex
\begin{picture}(12,5)
	\put(7,2.5){\circle*{0.6}}
	\put(7,1.5){\circle*{0.6}}
	\color{blue}
	\linethickness{0.33mm}
	\path(0,0)(7,1.5)(11,1)
	\path(1,4)(7,2.5)(9,5)
\end{picture}
}

\vmskip

Hence, $\mcc_1 = \{n \,|\, n \isdiv 2\pc\ n \gte 6 \}$.

\xxx

If the self-intersection index $k$ equals $2$, then the answer is 
$$
\mcc_2 = \{5\} \cup \{n \,|\, n \gte 7\} \pp
$$

Obviously, for $k=2$, $n < 5$ polylines of type \nck{n}{2} do not exist. Lemma $\ref{thm:nk62}$ tells us that the same is true for $n=6$ as well. For any odd $n$, beginning with five, an example of a type \nck{n}{2} polyline is quite simple---it is enough to take the vertices of a regular polygon with $n$ sides and connect them by diagonals of span~$2$. 

\BeginPic
\wtext{
\setlength{\unitlength}{2.8mm}
\begin{picture}(12,11)(-1,0)
	\put(1,5){\circle*{0.6}}
	\put(11,5){\circle*{0.6}}
	\put(3,0){\circle*{0.6}}
	\put(3,10){\circle*{0.6}}
	\put(9,0){\circle*{0.6}}
	\put(9,10){\circle*{0.6}}
	\put(6,6.5){\circle*{0.6}}
	\put(6,3.5){\circle*{0.6}}
	\color{blue}
	\linethickness{0.33mm}
	\path(1,5)(9,10)(6,3.5)(3,10)(11,5)(3,0)(6,6.5)(9,0)(1,5)
\end{picture}
}{w=0.24}

Further, for an arbitrary even $n \gte 10$ an example can be constructed by gluing together two copies of a type \nck{\frac n2}{2} polyline or of two polylines of types \nck{\frac{n}2-1}{2} and~\nck{\frac n2+1}{2}. Indeed, it is easy to see that any two polylines with the same self-intersection index can be glued together using the procedure we have already described.

Finally, an example of a type \nck{8}{2} polyline is shown in the figure. Therefore, polylines of type \nck{n}{2} exist if and only if $n = 5$ or~$n\gte 7$.\ftnnote{See also articles \cite{akk_2003} and~\cite{adbavg_2019}.}

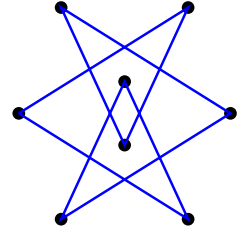
\EndPic

\section{Various construction procedures}
\hlabel{sec:ks}

Some of the polyline construction ideas we used in the previous section can be generalized for the cases of an arbitrary self-intersection index~$k$.

\vsskip

\fxkcb{1} Assuming $n \gte 2k+1$ and $\gcd(n,k)=1$, it is possible to construct a polyline of type~\nck{n}{2(k-1)}. Abandoning the formality, we can write it as:
$$
n \gte 2k+1 \pc\ \gcd(n,k)=1 \implies \nck{n}{2(k-1)} \pp
$$

Indeed, consider vertices of a regular polygon with $n$ sides and connect them by diagonals of span $k$ (that is, diagonals connecting pairs of vertices separated by exactly $k-1$ other vertices). Since $n$ and $k$ are co-prime, these segments form one closed polyline, each edge of which intersects exactly $2(k-1)$ other edges.

We will call these polylines \textit{star-sha\-pes} (or simply, \textit{stars}) of type $\wstar(n,k)$ (that is, $n$ vertices of a regular polygon connected by all diagonals of span~$k$).

\vsskip

\fxkcb{2} If there exists a polyline of type \nck{n}{pk}, then there exists a polyline of type~\nck{pn}{k}. In short,
$$
\nck{n}{pk} \implies \nck{pn}{k} \pp
$$

To prove that, consider some \nck{n}{pk} polyline. Clearly, we may assume that all self-intersection points are pairwise distinct. If that is not so, it would suffice to slightly tweak the vertices of the polyline while keeping the \nck{n}{pk} condition intact. Now, mark every edge of this polyline with $p-1$ ``ticks'', splitting all the self-intersection points on this edge into $p$ groups of $k$ points in each group. After that, ``break and slightly skew'' each edge at every tick. Then the obtained closed polyline will satisfy property~\nck{pn}{k}.

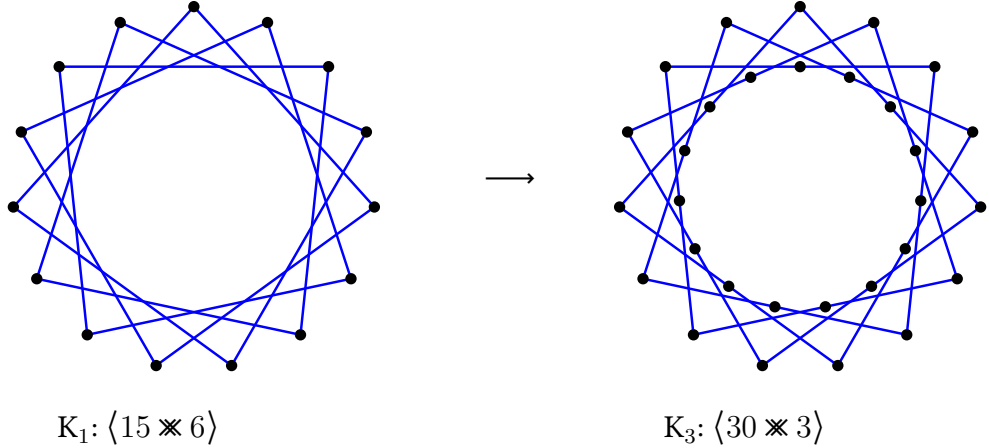
\vsskip

\hlabel{fig:k2_15_6}
\centerline{%
\setlength{\unitlength}{6mm}
\begin{picture}(10,10)(-5,-6)
  
  
  \color{blue}
  \linethickness{0.33mm}
  
  \qbezier(0,4)(1.990,1.791)(3.980,-0.418)
  \qbezier(1.624,3.649)(2.544,0.824)(3.464,-2.000)
  \qbezier(2.970,2.676)(2.660,-0.280)(2.351,-3.236)
  \qbezier(3.804,1.236)(2.318,-1.339)(0.832,-3.914)
  \qbezier(3.980,-0.418)(1.574,-2.166)(-0.832,-3.914)
  \qbezier(3.464,-2.000)(0.557,-2.618)(-2.351,-3.236)
  \qbezier(2.351,-3.236)(-0.557,-2.618)(-3.464,-2.000)
  \qbezier(0.832,-3.914)(-1.574,-2.166)(-3.980,-0.418)
  \qbezier(-0.832,-3.914)(-2.318,-1.339)(-3.804,1.236)
  \qbezier(-2.351,-3.236)(-2.660,-0.280)(-2.970,2.676)
  \qbezier(-3.464,-2.000)(-2.544,0.824)(-1.624,3.649)
  \qbezier(-3.980,-0.418)(-1.990,1.791)(0,4)
  \qbezier(-3.804,1.236)(-1.090,2.442)(1.624,3.649)
  \qbezier(-2.970,2.676)(0,2.676)(2.970,2.676)
  \qbezier(-1.624,3.649)(1.090,2.442)(3.804,1.236)
  
  \color{black}
  \put(0,4){\circle*{0.25}}
  \put(1.624,3.649){\circle*{0.25}}
  \put(2.970,2.676){\circle*{0.25}}
  \put(3.804,1.236){\circle*{0.25}}
  \put(3.980,-0.418){\circle*{0.25}}
  \put(3.464,-2.000){\circle*{0.25}}
  \put(2.351,-3.236){\circle*{0.25}}
  \put(0.832,-3.914){\circle*{0.25}}
  \put(-0.832,-3.914){\circle*{0.25}}
  \put(-2.351,-3.236){\circle*{0.25}}
  \put(-3.464,-2.000){\circle*{0.25}}
  \put(-3.980,-0.418){\circle*{0.25}}
  \put(-3.804,1.236){\circle*{0.25}}
  \put(-2.970,2.676){\circle*{0.25}}
  \put(-1.624,3.649){\circle*{0.25}}
  
  \put(-3,-5.5){\kcb{1}:\ \nck{15}{6}}
  
\end{picture}
\kern4ex
\begin{picture}(0,10)
	\put(0,6){$\longrightarrow$}
\end{picture}
\kern6ex
\begin{picture}(10,10)(-5,-6)
  
  \color{blue}
  \linethickness{0.33mm}
  
  \qbezier(0,4)(1.990,1.791)(3.980,-0.418)
  \qbezier(1.624,3.649)(2.544,0.824)(3.464,-2.000)
  \qbezier(2.970,2.676)(2.660,-0.280)(2.351,-3.236)
  \qbezier(3.804,1.236)(2.318,-1.339)(0.832,-3.914)
  \qbezier(3.980,-0.418)(1.574,-2.166)(-0.832,-3.914)
  \qbezier(3.464,-2.000)(0.557,-2.618)(-2.351,-3.236)
  \qbezier(2.351,-3.236)(-0.557,-2.618)(-3.464,-2.000)
  \qbezier(0.832,-3.914)(-1.574,-2.166)(-3.980,-0.418)
  \qbezier(-0.832,-3.914)(-2.318,-1.339)(-3.804,1.236)
  \qbezier(-2.351,-3.236)(-2.660,-0.280)(-2.970,2.676)
  \qbezier(-3.464,-2.000)(-2.544,0.824)(-1.624,3.649)
  \qbezier(-3.980,-0.418)(-1.990,1.791)(0,4)
  \qbezier(-3.804,1.236)(-1.090,2.442)(1.624,3.649)
  \qbezier(-2.970,2.676)(0,2.676)(2.970,2.676)
  \qbezier(-1.624,3.649)(1.090,2.442)(3.804,1.236)
  
  \color{black}
  \put(0,4){\circle*{0.25}}
  \put(1.624,3.649){\circle*{0.25}}
  \put(2.970,2.676){\circle*{0.25}}
  \put(3.804,1.236){\circle*{0.25}}
  \put(3.980,-0.418){\circle*{0.25}}
  \put(3.464,-2.000){\circle*{0.25}}
  \put(2.351,-3.236){\circle*{0.25}}
  \put(0.832,-3.914){\circle*{0.25}}
  \put(-0.832,-3.914){\circle*{0.25}}
  \put(-2.351,-3.236){\circle*{0.25}}
  \put(-3.464,-2.000){\circle*{0.25}}
  \put(-3.980,-0.418){\circle*{0.25}}
  \put(-3.804,1.236){\circle*{0.25}}
  \put(-2.970,2.676){\circle*{0.25}}
  \put(-1.624,3.649){\circle*{0.25}}
  
  \put(1.990,1.791){\circle*{0.25}}
  \put(2.544,0.824){\circle*{0.25}}
  \put(2.660,-0.280){\circle*{0.25}}
  \put(2.318,-1.339){\circle*{0.25}}
  \put(1.574,-2.166){\circle*{0.25}}
  \put(0.557,-2.618){\circle*{0.25}}
  \put(-0.557,-2.618){\circle*{0.25}}
  \put(-1.574,-2.166){\circle*{0.25}}
  \put(-2.318,-1.339){\circle*{0.25}}
  \put(-2.660,-0.280){\circle*{0.25}}
  \put(-2.544,0.824){\circle*{0.25}}
  \put(-1.990,1.791){\circle*{0.25}}
  \put(-1.090,2.442){\circle*{0.25}}
  \put(0,2.676){\circle*{0.25}}
  \put(1.090,2.442){\circle*{0.25}}
  
  \put(-3,-5.5){\kcb{3}:\ \nck{30}{3}}
  
\end{picture}
}

\vsskip

\fxkcb{3} Applying both procedures \kcb{1} and \kcb{2}, we can also obtain
$$
n \gte 2k+1 \pc\ \gcd(n,k)=1 \implies \nck{2n}{k-1} \pc
$$
or equivalently,
$$
n \gte 2k+3 \pc\ \gcd(n,k+1)=1 \implies \nck{2n}{k} \pp
$$

\vsskip

\fxkcb{4} A slightly better result can be achieved with a very similar procedure. Consider two regular polygons with $n$ sides, one of which is obtained from the other by homothety $H$ with sufficiently large ratio (it is enough to set ratio to~$n$). Choose now two non-negative integers $p, q$, both not exceeding $(n+1)/4$, and such that $\gcd(n, p+q) = 1$.

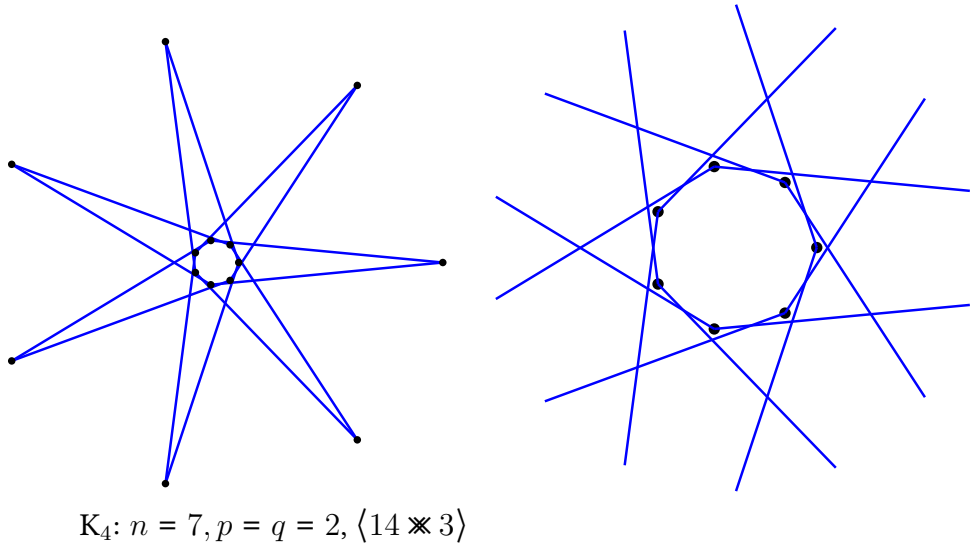
\vsskip

\centerline{%
\hlabel{k4ex1}
\setlength{\unitlength}{3.0mm}
\begin{picture}(20,25)
  \color{blue}
  \linethickness{0.33mm}
  \path(15.00, 14.00)(11.77, 23.74)
  \path(14.62, 14.78)(4.99, 18.33)
  \path(13.77, 14.97)(4.99, 9.66)
  \path(13.09, 14.43)(11.77, 4.25)
  \path(13.09, 13.56)(20.23, 6.18)
  \path(13.77, 13.02)(24.00, 14.00)
  \path(14.62, 13.21)(20.23, 21.81)
  \path(24.00, 14.00)(13.77, 14.97)
  \path(20.23, 21.81)(13.09, 14.43)
  \path(11.77, 23.74)(13.09, 13.56)
  \path( 4.99, 18.33)(13.77, 13.02)
  \path( 4.99,  9.66)(14.62, 13.21)
  \path(11.77,  4.25)(15.00, 14.00)
  \path(20.23,  6.18)(14.62, 14.78)
  %
  \color{black}
  \put(15.00, 14.00){\circle*{0.33}}
  \put(14.62, 14.78){\circle*{0.33}}
  \put(13.77, 14.97){\circle*{0.33}}
  \put(13.09, 14.43){\circle*{0.33}}
  \put(13.09, 13.56){\circle*{0.33}}
  \put(13.77, 13.02){\circle*{0.33}}
  \put(14.62, 13.21){\circle*{0.33}}
  %
  %
  \put(24.00, 14.00){\circle*{0.33}}
  \put(20.23, 21.81){\circle*{0.33}}
  \put(11.77, 23.74){\circle*{0.33}}
  \put( 4.99, 18.33){\circle*{0.33}}
  \put( 4.99,  9.66){\circle*{0.33}}
  \put(11.77,  4.25){\circle*{0.33}}
  \put(20.23,  6.18){\circle*{0.33}}
  \put(8,2){\kcb{4}: $n=7, p=q=2$, \nck{14}{3}}
\end{picture}
\kern3ex
\begin{picture}(30,25)
  \put(18.33, 14.66){\circle*{0.5}}
  \put(16.93, 17.53){\circle*{0.5}}
  \put(13.81, 18.23){\circle*{0.5}}
  \put(11.33, 16.24){\circle*{0.5}}
  \put(11.33, 13.05){\circle*{0.5}}
  \put(13.81, 11.07){\circle*{0.5}}
  \put(16.93, 11.77){\circle*{0.5}}
  \color{blue}
  \linethickness{0.33mm}
  \path(18.33, 14.66)(14.77, 25.37)
  \path(25.08, 17.16)(13.81, 18.23)
  \path(16.93, 17.53)(06.34, 21.45)
  \path(19.17, 24.34)(11.33, 16.24)
  \path(13.81, 18.23)(04.18, 12.39)
  \path( 9.86, 24.23)(11.33, 13.05)
  \path(11.33, 16.24)(09.86, 05.06)
  \path( 4.18, 16.90)(13.81, 11.07)
  \path(11.33, 13.05)(19.17, 04.95)
  \path( 6.34, 07.88)(16.93, 11.77)
  \path(13.81, 11.07)(25.08, 12.13)
  \path(14.77, 03.92)(18.33, 14.66)
  \path(16.93, 11.77)(23.10, 21.23)
  \path(23.10, 08.06)(16.93, 17.53)
\end{picture}
}

\vmskip

Begin with indexing all the vertices of the larger $n$-gon in counterclockwise order, denoting them $A_0$, $A_1$, ..., $A_{n-1}$, and treating these indices as residues modulo~$n$. Denote the corresponding vertices of the smaller $n$-gon by $B_0$, $B_1$, ..., and~$B_{n-1}$. 

Now, for every residue $i$ modulo $n$ connect point $B_i$ with point~$A_{(i+p)\mmod n}$, and every point $A_i$---with point~$B_{(i+q)\mmod n}$. It is clear that all these segments form one closed polyline of type~\nck{2n}{p+q-1}. The figure above shows an example of \nck{14}{3} polyline (for $n = 7$, $p=q=2$,~$k=3$).

If we set $k = p+q-1$, then
$$
n \gte 2k+1 \pc\ \gcd(n,k+1)=1 \implies \nck{2n}{k} \pp
$$
In other words, comparing this with procedure \kcb{3}, we obtain one additional polyline with a fixed self-intersection index. Namely, for $n = 2k+1$, we are getting a polyline of type~\nck{4k+2}{k}.

\xxx

Another interesting result can be achieved by varying the homothety ratio~$h$. Let us set $p=q$ and change the restriction $p \lte (n+1)/4$ to $p \lte n/2$. Then for $\gcd(n, p) = 1$ and more or less arbitrary $h \gte 1$, as long as the points of our collection are in a general position, we will obtain polyline \pdph{n}{p}{h} with $2n$ vertices and constant self-intersection index.

Due to the obvious rotational and axial symmetry, that index will always be either zero or an odd number.

The figure below has two examples of polylines \pdph{7}{2}{1.5} (type \nck{14}{5}) and \pdph{7}{3}{2.5} (type~\nck{14}{9}).

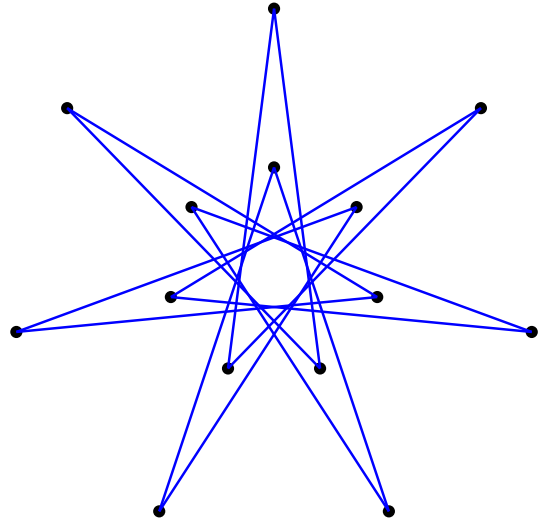
\vtskip

\centerline{%
\hlabel{fig:k4ex2}
\setlength{\unitlength}{7mm}
\begin{picture}(10,12)
  \color{black}
  \put(8.12, 8.91){\circle*{0.23}}
  \put(4.26, 1.75){\circle*{0.23}}
  \put(3.89, 9.87){\circle*{0.23}}
  \put(7.08, 2.39){\circle*{0.23}}
  \put(0.50, 7.17){\circle*{0.23}}
  \put(8.33, 5.00){\circle*{0.23}}
  \put(0.50, 2.83){\circle*{0.23}}
  \put(7.08, 7.61){\circle*{0.23}}
  \put(3.89, 0.13){\circle*{0.23}}
  \put(4.26, 8.25){\circle*{0.23}}
  \put(8.12, 1.09){\circle*{0.23}}
  \put(2.00, 6.45){\circle*{0.23}}
  \put(10.00, 5.00){\circle*{0.23}}
  \put(2.00, 3.55){\circle*{0.23}}
  \put(0,11){\pdph{7}{2}{1.5} \pc\ type \nck{14}{5}}
  %
  \color{blue}
  \linethickness{0.33mm}
  \path(8.12, 8.91)(7.08, 2.39)
  \path(3.89, 9.87)(8.33, 5.00)
  \path(0.50, 7.17)(7.08, 7.61)
  \path(0.50, 2.83)(4.26, 8.25)
  \path(3.89, 0.13)(2.00, 6.45)
  \path(8.12, 1.09)(2.00, 3.55)
  \path(10.00, 5.00)(4.26, 1.75)
  \path(8.12, 8.91)(2.00, 6.45)
  \path(3.89, 9.87)(2.00, 3.55)
  \path(0.50, 7.17)(4.26, 1.75)
  \path(0.50, 2.83)(7.08, 2.39)
  \path(3.89, 0.13)(8.33, 5.00)
  \path(8.12, 1.09)(7.08, 7.61)
  \path(10.00, 5.00)(4.26, 8.25)
\end{picture}
\kern3ex
\begin{picture}(10,12)
  \color{black}
  \put(2.83, 0.50){\circle*{0.23}}
  \put(4.13, 3.20){\circle*{0.23}}
  \put(7.17, 0.50){\circle*{0.23}}
  \put(5.87, 3.20){\circle*{0.23}}
  \put(9.87, 3.89){\circle*{0.23}}
  \put(6.95, 4.55){\circle*{0.23}}
  \put(8.91, 8.12){\circle*{0.23}}
  \put(6.56, 6.25){\circle*{0.23}}
  \put(5.00, 10.00){\circle*{0.23}}
  \put(5.00, 7.00){\circle*{0.23}}
  \put(1.09, 8.12){\circle*{0.23}}
  \put(3.44, 6.25){\circle*{0.23}}
  \put(0.13, 3.89){\circle*{0.23}}
  \put(3.05, 4.55){\circle*{0.23}}
  \put(0,11){\pdph{7}{3}{2.5} \pc\ type \nck{14}{9}}
  %
  \color{blue}
  \linethickness{0.33mm}
  \path(2.83, 0.50)(6.56, 6.25)
  \path(7.17, 0.50)(5.00, 7.00)
  \path(9.87, 3.89)(3.44, 6.25)
  \path(8.91, 8.12)(3.05, 4.55)
  \path(5.00, 10.00)(4.13, 3.20)
  \path(1.09, 8.12)(5.87, 3.20)
  \path(0.13, 3.89)(6.95, 4.55)
  \path(2.83, 0.50)(5.00, 7.00)
  \path(7.17, 0.50)(3.44, 6.25)
  \path(9.87, 3.89)(3.05, 4.55)
  \path(8.91, 8.12)(4.13, 3.20)
  \path(5.00, 10.00)(5.87, 3.20)
  \path(1.09, 8.12)(6.95, 4.55)
  \path(0.13, 3.89)(6.56, 6.25)
\end{picture}
}

\vsskip

\fxkcb{5} The ``gluing'' surgery mentioned above shows that if types \nck{m}{k} and \nck{n}{k} are feasible, then so is the type~\nck{m+n}{k}. 

This can be informally written as
$$
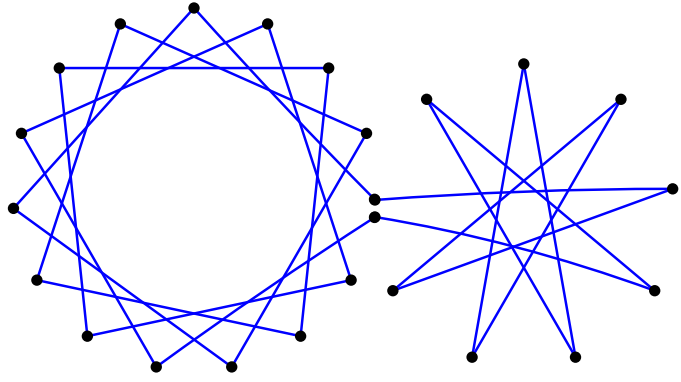
\nck{m}{k} \pc\ \nck{n}{k} \implies \nck{m+n}{k} \pp
$$

\vmskip

\centerline{%
\setlength{\unitlength}{6mm}
\begin{picture}(10,10)(-5,-6)
  
  
  \color{blue}
  \linethickness{0.33mm}
  
  \qbezier(0,4)(1.990,1.791)(3.980,-0.218)
  \qbezier(1.624,3.649)(2.544,0.824)(3.464,-2.000)
  \qbezier(2.970,2.676)(2.660,-0.280)(2.351,-3.236)
  \qbezier(3.804,1.236)(2.318,-1.339)(0.832,-3.914)
  \qbezier(3.980,-0.618)(1.574,-2.166)(-0.832,-3.914)
  \qbezier(3.464,-2.000)(0.557,-2.618)(-2.351,-3.236)
  \qbezier(2.351,-3.236)(-0.557,-2.618)(-3.464,-2.000)
  \qbezier(0.832,-3.914)(-1.574,-2.166)(-3.980,-0.418)
  \qbezier(-0.832,-3.914)(-2.318,-1.339)(-3.804,1.236)
  \qbezier(-2.351,-3.236)(-2.660,-0.280)(-2.970,2.676)
  \qbezier(-3.464,-2.000)(-2.544,0.824)(-1.624,3.649)
  \qbezier(-3.980,-0.418)(-1.990,1.791)(0,4)
  \qbezier(-3.804,1.236)(-1.090,2.442)(1.624,3.649)
  \qbezier(-2.970,2.676)(0,2.676)(2.970,2.676)
  \qbezier(-1.624,3.649)(1.090,2.442)(3.804,1.236)
  
  \color{black}
  \put(0,4){\circle*{0.25}}
  \put(1.624,3.649){\circle*{0.25}}
  \put(2.970,2.676){\circle*{0.25}}
  \put(3.804,1.236){\circle*{0.25}}
  \put(3.980,-0.218){\circle*{0.25}}
  \put(3.980,-0.618){\circle*{0.25}}
  \put(3.464,-2.000){\circle*{0.25}}
  \put(2.351,-3.236){\circle*{0.25}}
  \put(0.832,-3.914){\circle*{0.25}}
  \put(-0.832,-3.914){\circle*{0.25}}
  \put(-2.351,-3.236){\circle*{0.25}}
  \put(-3.464,-2.000){\circle*{0.25}}
  \put(-3.980,-0.418){\circle*{0.25}}
  \put(-3.804,1.236){\circle*{0.25}}
  \put(-2.970,2.676){\circle*{0.25}}
  \put(-1.624,3.649){\circle*{0.25}}
  
  \put(-4,-5.5){Example:~~\kcb{5} \nck{15}{6}, \nck{9}{6} $\implies$ \nck{24}{6}}
  
\end{picture}
\kern-7.45ex
\setlength{\unitlength}{5mm}
\begin{picture}(10,10)(-5,-6.52)
  
  
  \color{blue}
  \linethickness{0.33mm}
  
  \qbezier(0,4)(0.684,0.121)(1.368,-3.759)
  \qbezier(2.571,3.064)(0.602,-0.347)(-1.368,-3.759)
  \qbezier(3.939,0.695)(0.238,-0.653)(-3.464,-2.000)
  \qbezier(3.464,-2.000)(-0.238,-0.653)(-3.939,-0.050)
  \qbezier(1.368,-3.759)(-0.602,-0.347)(-2.571,3.064)
  \qbezier(-1.368,-3.759)(-0.684,0.121)(0,4)
  \qbezier(-3.464,-2.000)(-0.447,0.532)(2.571,3.064)
  \qbezier(-3.939,0.400)(0,0.695)(3.939,0.695)
  \qbezier(-2.571,3.064)(0.447,0.532)(3.464,-2.000)
  
  \color{black}
  \put(0,4){\circle*{0.3}}
  \put(2.571,3.064){\circle*{0.3}}
  \put(3.939,0.695){\circle*{0.3}}
  \put(3.464,-2.000){\circle*{0.3}}
  \put(1.368,-3.759){\circle*{0.3}}
  \put(-1.368,-3.759){\circle*{0.3}}
  \put(-3.464,-2.000){\circle*{0.3}}
  \put(-3.939,-0.050){\circle*{0.3}}
  \put(-3.939,0.400){\circle*{0.3}}
  \put(-2.571,3.064){\circle*{0.3}}
 
\end{picture}
}

\vsskip

\fxkcb{6} If type \nck{n}{k} is feasible, then for any natural $p$ there exists a polyline of type~\nck{pn}{pk}. That is,
$$
\nck{n}{k} \implies \nck{pn}{pk} \pp
$$
Indeed, consider some polyline $L$ of type \nck{n}{k} and choose a direction for traversing that polyline. Then any oriented edge $\vec e$ of the polyline can be shifted by a vector $\vec v$ of small length $\veps$ perpendicular to $\vec e$ and such that vector pair $(\vec e, \vec v)$ is positively oriented. Extending or shrinking these new line segments near the vertices of $L$, we obtain new polyline~$L_{\veps}$.

\wtext{%
\setlength{\unitlength}{3mm}
\begin{picture}(13,11)(-1,-1)
	\put(-0.8,-0.3){\circle*{0.5}}
	\put(9,2){\circle*{0.5}}
	\put(6,11){\circle*{0.5}}
	\put(3,2){\circle*{0.5}}
	\put(12.8,-0.3){\circle*{0.5}}
	\put(6,7){\circle*{0.5}}
	\put(1.1,0.7){\circle*{0.5}}
	\put(8.0,2.5){\circle*{0.5}}
	\put(6,9.0){\circle*{0.5}}
	\put(3.9,2.5){\circle*{0.5}}
	\put(11.0,0.7){\circle*{0.5}}
	\put(6,6.9){\circle*{0.5}}
    \color{blue}
    \linethickness{0.33mm}
	\path(6,7)(-0.8,-0.3)(9,2)(6,11)(3,2)(12.8,-0.3)(6,7)
	\path(1.1,0.7)(8.0,2.5)(6,9.0)(3.9,2.5)(11.0,0.7)(6,6.0)(1.1,0.7)
\end{picture}
}{w=0.26,v=0ex,l=9}

Build $p$ copies of this polyline, obtained from $L$ by such shifts using parameters $\veps$, $2\veps$, \ldots, $(p-1)\veps$ for a very small positive value of~$\veps$ (see an example for $p=2$ in the figure). Obviously, each vertex of original polyline turned into $p$ vertices, while every edge turned into $p$ edges. It is quite evident that every edge now intersects exactly $pk$ other edges. It remains to accurately merge all these $p$ ``parallel'' polylines---let us call them $L_1 = L$, $L_2$, \ldots,~$L_{p}$---into one closed polyline. 

First, consider case $p=2$. Choose an arbitrary vertex of the original polyline~$L_1$, as well as the corresponding vertex of the second polyline~$L_2$. Then the resolution ``surgery'', shown in the figure below, allows us to merge these two polylines into one closed polyline of type~\nck{2n}{2k}.

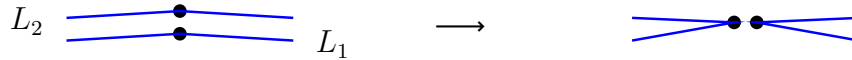
\vlskip

\centerline{%
\setlength{\unitlength}{3mm}
\begin{picture}(14,1.5)(-3,0)
	\put(5,0.3){\circle*{0.6}}
	\put(5,1.3){\circle*{0.6}}
	\put(11,-0.5){$L_1$}
	\put(-2.5,0.5){$L_2$}
    \color{blue}
    \linethickness{0.33mm}
	\path(0,0)(5,0.3)(10,0)
	\path(0,1)(5,1.3)(10,1)
\end{picture}
\kern8ex
\begin{picture}(3,1.5)
	\put(0,0.2){$\longrightarrow$}
\end{picture}
\kern4ex
\begin{picture}(14,1.5)(-3,0)
	\put(4.5,0.8){\circle*{0.6}}
	\put(5.5,0.8){\circle*{0.6}}
%
%
    \color{blue}
    \linethickness{0.33mm}
	\path(0,0)(4.5,0.8)(0,1)
	\path(10,0)(5.5,0.8)(10,1)
\end{picture}
}

\vmskip

Second, we need to demonstrate a similar merge for any odd value of $p = 2q+1$. To do that, consider an arbitrary vertex of the original polyline, as well as all $p-1$ corresponding additional vertices of the parallel polylines. After this, slightly shift the incident edges to perform the polyline surgery shown below. This will merge all the polylines in pairs $L_1$ and $L_2$, \ldots, $L_{2q-1}$ and~$L_{2q}$. See below an example for~$p=5$.

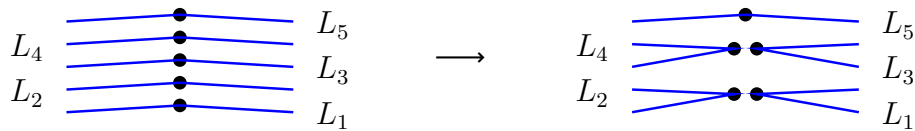
\vmskip

\centerline{%
\setlength{\unitlength}{3mm}
\begin{picture}(14,4)(-3,0)
	\put(5,0.3){\circle*{0.6}}
	\put(5,1.3){\circle*{0.6}}
	\put(5,2.3){\circle*{0.6}}
	\put(5,3.3){\circle*{0.6}}
	\put(5,4.3){\circle*{0.6}}
	\put(11,-0.5){$L_1$}
	\put(-2.5,0.5){$L_2$}
	\put(11,1.5){$L_3$}
	\put(-2.5,2.5){$L_4$}
	\put(11,3.5){$L_5$}
    \color{blue}
    \linethickness{0.33mm}
	\path(0,0)(5,0.3)(10,0)
	\path(0,1)(5,1.3)(10,1)
	\path(0,2)(5,2.3)(10,2)
	\path(0,3)(5,3.3)(10,3)
	\path(0,4)(5,4.3)(10,4)
\end{picture}
\kern8ex
\begin{picture}(3,5)
	\put(0,2){$\longrightarrow$}
\end{picture}
\kern4ex
\begin{picture}(14,4)(-3,0)
	\put(4.5,0.8){\circle*{0.6}}
	\put(5.5,0.8){\circle*{0.6}}
	\put(4.5,2.8){\circle*{0.6}}
	\put(5.5,2.8){\circle*{0.6}}
	\put(5,4.3){\circle*{0.6}}
	\put(11,-0.5){$L_1$}
	\put(-2.5,0.5){$L_2$}
	\put(11,1.5){$L_3$}
	\put(-2.5,2.5){$L_4$}
	\put(11,3.5){$L_5$}
    \color{blue}
    \linethickness{0.33mm}
	\path(0,0)(4.5,0.8)(0,1)
	\path(10,0)(5.5,0.8)(10,1)
	\path(0,2)(4.5,2.8)(0,3)
	\path(10,2)(5.5,2.8)(10,3)
	\path(0,4)(5,4.3)(10,4)
\end{picture}
}

\vmskip

After that, perform an identical operation for another vertex of $L_1 = L$, but this time connect the polylines in pairs $L_2$ and $L_3$, \ldots, $L_{2q}$ and~$L_{2q+1}$.

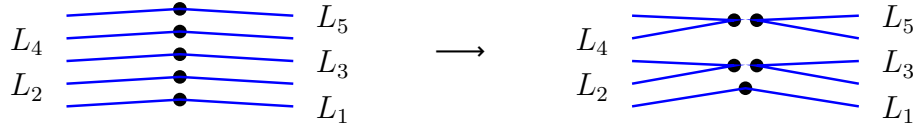
\vlskip

\centerline{%
\setlength{\unitlength}{3mm}
\begin{picture}(14,4)(-3,0)
	\put(5,0.3){\circle*{0.6}}
	\put(5,1.3){\circle*{0.6}}
	\put(5,2.3){\circle*{0.6}}
	\put(5,3.3){\circle*{0.6}}
	\put(5,4.3){\circle*{0.6}}
	\put(11,-0.5){$L_1$}
	\put(-2.5,0.5){$L_2$}
	\put(11,1.5){$L_3$}
	\put(-2.5,2.5){$L_4$}
	\put(11,3.5){$L_5$}
    \color{blue}
    \linethickness{0.33mm}
	\path(0,0)(5,0.3)(10,0)
	\path(0,1)(5,1.3)(10,1)
	\path(0,2)(5,2.3)(10,2)
	\path(0,3)(5,3.3)(10,3)
	\path(0,4)(5,4.3)(10,4)
\end{picture}
\kern8ex
\begin{picture}(3,5)
	\put(0,2){$\longrightarrow$}
\end{picture}
\kern4ex
\begin{picture}(14,4)(-3,0)
	\put(4.5,1.8){\circle*{0.6}}
	\put(5.5,1.8){\circle*{0.6}}
	\put(4.5,3.8){\circle*{0.6}}
	\put(5.5,3.8){\circle*{0.6}}
	\put(5,0.8){\circle*{0.6}}
	\put(11,-0.5){$L_1$}
	\put(-2.5,0.5){$L_2$}
	\put(11,1.5){$L_3$}
	\put(-2.5,2.5){$L_4$}
	\put(11,3.5){$L_5$}
    \color{blue}
    \linethickness{0.33mm}
	\path(0,0)(5,0.8)(10,0)
	\path(0,1)(4.5,1.8)(0,2)
	\path(10,1)(5.5,1.8)(10,2)
	\path(0,3)(4.5,3.8)(0,4)
	\path(10,3)(5.5,3.8)(10,4)
\end{picture}
}

\vmskip

As a result of these two surgeries, all polylines $L_i$ will be merged into one closed polyline of type~\nck{pn}{pk}. 

For any odd value of $p$, there exists another, somewhat simpler procedure which also proves implication $\nck{n}{k} \implies \nck{pn}{pk}$. Namely, every edge in a polyline of type \nck{n}{k} can be replaced by a $p$-fold ``lightning'' (see below).

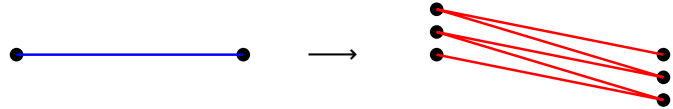
\vmskip

\centerline{%
\setlength{\unitlength}{3mm}
\begin{picture}(10,4)(0,-1)
	\put(0,1){\circle*{0.6}}
	\put(10,1){\circle*{0.6}}
    \color{blue}
    \linethickness{0.33mm}
	\polyline(0,1)(10,1)
\end{picture}
\kern4ex
\begin{picture}(3,4)(0,-1)
	\put(0,0.6){$\longrightarrow$}
\end{picture}
\kern4ex
\begin{picture}(10,4)(0,-1)
	\put(0.0,1.0){\circle*{0.6}}
	\put(0.0,2.0){\circle*{0.6}}
	\put(0.0,3.0){\circle*{0.6}}
	\put(10.0,-1.0){\circle*{0.6}}
	\put(10.0,0.0){\circle*{0.6}}
	\put(10.0,1.0){\circle*{0.6}}
    \color{red}
    \linethickness{0.33mm}
	\path(0,1)(10,-1)(0,2)(10,0)(0,3)(10,1)
\end{picture}
}

\vsskip

\noindent and that surgery, obviously, creates a polyline of type~\nck{pn}{pk}.

\vsskip

\fxkcb{7} Executing procedures \kcb{6} and \kcb{5} proves that if types \nck{m}{p} and \nck{n}{q} are feasible, then so is type \nck{mq + np}{pq}.
$$
\nck{m}{p}, \nck{n}{q} \implies \nck{mq + np}{pq} \pp
$$
In particular, since for $p\gte 3$ there exists a polyline of type \nck{2p}{1}, we obtain
$$
\nck{m}{n}, p\gte 3 \implies \nck{m + 2pn}{n} \pp
$$

This particular corollary of procedure \kcb{7} can be strengthened.

\vsskip

\fxkcb{8} Let $k$ be even and assume that there exists a polyline of type~\nck{n}{k}. Consider an arbitrary edge $\tau$ of this polyline and replace it with $(k+1)$-fold lightning; this can be done since $k+1$ is odd. Each edge of the lightning intersects exactly $k$ other edges of the polyline. Obviously, the total number of edges has increased by $k$, and the number of self-intersection points on each of the edges $e_1$, ..., $e_k$ which intersected~$\tau$, has increased by $k$ and became~$2k$; all other edges were not affected by this operation. ``Breaking and skewing'' each edge $e_i$ at the new vertex which divides its $2k$ self-intersection points in two equal groups of size~$k$, we obtain the new polyline with $n + 2k$ edges and the same self-intersection index equal to~$k$.

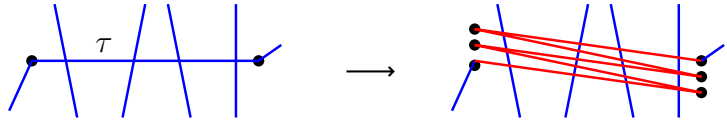
\vmskip

\centerline{%
\setlength{\unitlength}{3mm}
\begin{picture}(11,6)(0,-2)
	\put(0,1.5){\circle*{0.5}}
	\put(10,1.5){\circle*{0.5}}
	\put(2.8,1.9){$\tau$}
    \color{blue}
    \linethickness{0.33mm}
	\polyline(-1,-0.7)(0,1.5)(10,1.5)(11,2.2)
	\polyline(2,-1)(1,4)
	\polyline(4,-1)(5,4)
	\polyline(7,-1)(6,4)
	\polyline(9,-1)(9,4)
\end{picture}
\kern4ex
\begin{picture}(3,6)(0,-2)
	\put(0,0.6){$\longrightarrow$}
\end{picture}
\kern4ex
\begin{picture}(11,6)(0,-2)
	\put(0.0,1.3){\circle*{0.5}}
	\put(0.0,2.2){\circle*{0.5}}
	\put(0.0,2.9){\circle*{0.5}}
	\put(10.0,0.1){\circle*{0.5}}
	\put(10.0,0.8){\circle*{0.5}}
	\put(10.0,1.5){\circle*{0.5}}
    \color{blue}
    \linethickness{0.33mm}
	\polyline(2,-1)(1,4)
	\polyline(4,-1)(5,4)
	\polyline(7,-1)(6,4)
	\polyline(9,-1)(9,4)
	\polyline(-1,-0.7)(0,1.5)
	\polyline(10,1.5)(11,2.2)
    \color{red}
	\polyline(0,1.5)(10,0.1)(0,2.2)(10,0.8)(0,2.9)(10,1.5)
\end{picture}
}

\vsskip

If $k$ is odd, then we can replace an arbitrary edge $\tau$ with $(2k+1)$-fold lightning, increasing the total number of edges by~$2k$. After that, on every edge intersecting $\tau$, the number of self-intersection points becomes equal $(k-1)+(2k+1) = 3k$. After we ``break and skew'' each of these edges in two respective points, the result will be a polyline of type~\nck{n + 4k}{k}.

Hence,
\begin{align*}
\text{for even } k \pc & \quad \nck{n}{k} \implies \nck{n + 2k}{k} \\
\text{for odd } k \pc & \quad \nck{n}{k} \implies \nck{n + 4k}{k} \pp
\end{align*}

\vsskip

\fxkcb{9} Let $k > 1$ be an odd number such that there exists a polyline $\mcl$ of type~\nck{n}{k}. Now assume that one of the self-intersection points on $\mcl$ (we will call it $X$; let $\tau_1 = AB$ and $\tau_2 = MN$ be the edges that intersect at~$X$) satisfies the following two conditions:

\wtext{%
\setlength{\unitlength}{3mm}
\begin{picture}(10,9)(-1,0)
	\put(0.0, 5.50){\circle*{0.5}} 
	\put(8.0, 4.7){\circle*{0.5}}  
	\put(4.0, 0.0){\circle*{0.5}}  
	\put(8.0, 7.0){\circle*{0.5}}  
	\put(0.0, 6.2){$A$}
	\put(8.0, 2.9){$B$}
	\put(1.8, 0.0){$M$}
	\put(7.5, 7.6){$N$}
	\put(5.5, 5.3){$X$}
	\color{blue}
	\linethickness{0.33mm}
	\path(8.0, 4.7)(0.0, 5.5)
	\path(4.0, 0.0)(8.0, 7.0)
	\path(5.3, 1.0)(4.1, 1.2)
	\path(5.7, 2.0)(4.9, 2.8)
	\path(4.3, 4.2)(4.4, 5.6)
	\path(2.2, 4.6)(2.0, 6.1)
\end{picture}}{w = 0.22}

(\kcb{9}/$1$) Point $X$ is the \textit{marginal} point\ftnnote{We will call a self-intersection point $X$ \textit{marginal} on edge~$e$, if $X$ is either the first or the last point of self-intersection on~$e$.} of self-intersection on each of the edges $\tau_i$; let us assume that vertex $B$ is its neighbor on $\tau_1$, while vertex $N$ is its neighbor on~$\tau_2$ (see the figure);

(\kcb{9}/$2$) If edges $AB$ and $MN$ are replaced by edges $AN$ and $BM$, polyline $\mcl$ remains connected. That is equivalent to the following property: when traversing polyline $\mcl$, orientations of the edges $AB$ and $MN$ (in this specific order) must be opposite. In other words: if $\mcl$ is traversed starting from $A \to B$, then $N$ is the first vertex of edge $\tau_2 = MN$ we would visit.

Under these assumptions, there exists a polyline of type~\nck{n+2k}{k}. That is,
$$
\nck{n}{k} \implies \nck{n+2k}{k} \pp
$$

\vtskip

Indeed, denote the edges following $AB$ and $MN$ by $BC$ and~$NP$. Now, duplicate vertices $B$ and $N$, creating vertices $B'$ and $N"$ in such a way that edges $BC$ and~$NP$ turn into edges $B'C$ and $N'P$; thereupon we insert two $k$-fold lightnings connecting pairs ($B$, $N'$) and ($B'$, $N$) respectively---see the figure.

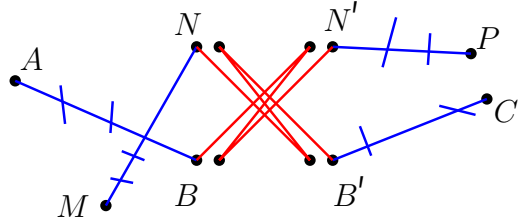
\centerline{%
\setlength{\unitlength}{3mm}
\begin{picture}(23,10)
	\put(0.0, 5.5){\circle*{0.5}} 
	\put(8.0, 2.0){\circle*{0.5}} 
	\put(4.0, 0.0){\circle*{0.5}} 
	\put(8.0, 7.0){\circle*{0.5}} 
	\put(14.0, 2.0){\circle*{0.5}} 
	\put(14.0, 7.0){\circle*{0.5}} 
	\put(20.8, 4.7){\circle*{0.5}} 
	\put(20.1, 6.7){\circle*{0.5}} 
	\put(13.0, 2.0){\circle*{0.5}}
	\put(9.0, 7.0){\circle*{0.5}}
	\put(13.0, 7.0){\circle*{0.5}}
	\put(9.0, 2.0){\circle*{0.5}}
	\put(0.2, 6.0){$A$}
	\put(7.0, 0.0){$B$}
	\put(1.8, -0.5){$M$}
	\put(7.0, 7.5){$N$}
	\put(14.0, 0.0){$B'$}
	\put(13.6, 7.7){$N'$}
	\put(21.1, 3.7){$C$}
	\put(20.3, 6.9){$P$}
	\color{blue}
	\linethickness{0.33mm}
	\path(8.0, 2.0)(0.0, 5.5)
	\path(4.0, 0.0)(8.0, 7.0)
	\path(14.0, 2.0)(20.8, 4.7)
	\path(14.0, 7.0)(20.1, 6.7)
	\path(2.2, 3.6)(2.0, 5.3) 
	\path(4.2, 3.2)(4.3, 4.6) 
	\path(4.2, 1.2)(5.2, 1.0) 
	\path(4.7, 2.4)(5.7, 2.0) 
	\path(16.2, 6.1)(16.8, 8.3) 
	\path(18.2, 6.2)(18.3, 7.6) 
	\path(15.2, 3.5)(15.7, 2.2) 
	\path(18.7, 4.4)(20.3, 4.0) 
	\color{red}
	\path(8.0, 7.0)(13.0, 2.0)(9.0, 7.0)(14.0, 2.0)
	\path(8.0, 2.0)(13.0, 7.0)(9.0, 2.0)(14.0, 7.0)
\end{picture}}

\vmskip

The result of this surgery is polyline $\mcl'$ with $n+2k$ edges where every edge is intersected by exactly $k$ other edges which is what we need. It follows from condition (\kcb{9}/$2$) that the new polyline consists of one closed component.

\vtskip

\note. Note that every \nck{n}{k} polyline possesses at least one self-intersection point which satisfies condition~(\kcb{9}/$1$). In fact, there are at least three such points. As an example, consider a self-intersection point with maximum abscissa. Obviously, it must be a marginal point on both edges to which it belongs. However, not every such point also satisfies condition~(\kcb{9}/$2$)---e.g., numerous counterexamples can be found on polylines $\mcz_{12,3}$, $\mcz_{14,3}$ etc from the next section.

However, on any polyline obtained by applying procedure \kcb{9}, such a point exists, and so procedure \kcb{9} can be repeated as many times as needed. Indeed, on the illustration shown above, the rightmost intersection point of two inserted $k$-fold lightnings satisfies both conditions (\kcb{9}/$1$) and~(\kcb{9}/$2$).

\vtskip

\begin{question}
Does every \nck{n}{k} polyline contain at least one self-intersection point which satisfies both conditions (\kcb{9}/$1$) and~(\kcb{9}/$2$)? 
\end{question}

Generally speaking, we are only interested in the case of odd $k$, but it is quite possible that the parity of $k$ does not affect the answer to this question.

\vsskip

\fxkcb{10} For any natural $k$, type \nck{4k+4}{2k} is feasible. The respective polyline can be constructed as follows.

On line $y=-1$, for each value of index $i$ from $1$ to $k+1$ mark point $A_i = (-i, -1)$; on line $y=-2$, for the same values of index $i$ mark points $B_i = (i+1, -2)$; on line $x=1$, for indices $i$ from $1$ to $2k+1$ mark points $C_i = (1, i)$; and, finally, add one more point~$O = (0,0)$. The resulting $4k+4$ points are then connected by segments from the following list:
\begin{align*}
C_1B_{k+1}     &\pc\ B_{k+1}A_{k+1} \pc\ A_{k+1}O \pc\  OB_{2k+1} \,; \\
C_{2i-2k+2}A_i &\pc\ A_iC_{2i-2k+1} \pc\  i = 1, ..., k \,; \\
C_{2i-2k+3}B_i &\pc\ B_iC_{2i-2k+2} \pc\  i = 1, ..., k \pp
\end{align*}
These segments form a polyline with $4k+4$ edges and self-intersection index~$2k$. It might seem that a polyline of type \nck{4k+4}{2k} could be constructed using procedure \kcb{1}, if it were not for the fact that $\gcd(4k+4,k+1) = k+1 > 1$.

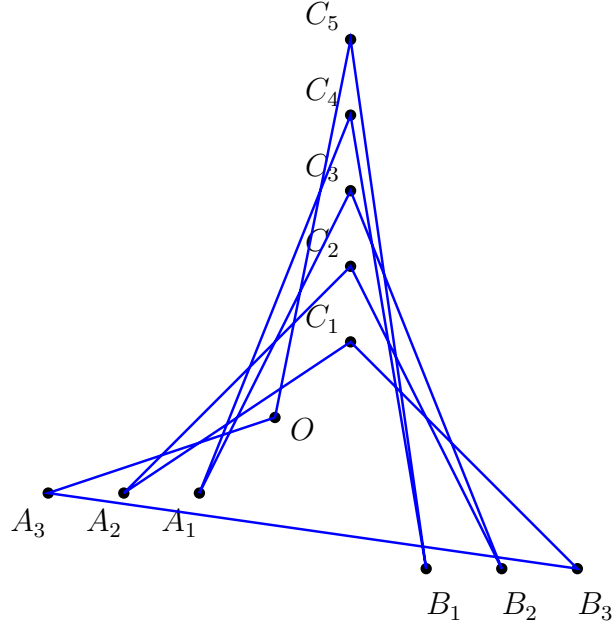
\vmskip

\centerline{%
\setlength{\unitlength}{10mm}
\begin{picture}(8,8.5)
	\put(4.0, 3.0){\circle*{0.15}}
	\put(5.0, 8.0){\circle*{0.15}}
	\put(6.0, 1.0){\circle*{0.15}}
	\put(5.0, 7.0){\circle*{0.15}}
	\put(3.0, 2.0){\circle*{0.15}}
	\put(5.0, 6.0){\circle*{0.15}}
	\put(7.0, 1.0){\circle*{0.15}}
	\put(5.0, 5.0){\circle*{0.15}}
	\put(2.0, 2.0){\circle*{0.15}}
	\put(5.0, 4.0){\circle*{0.15}}
	\put(8.0, 1.0){\circle*{0.15}}
	\put(1.0, 2.0){\circle*{0.15}}
	\put(4.2, 2.7){$O$}
	\put(2.5, 1.5){$A_1$}
	\put(1.5, 1.5){$A_2$}
	\put(0.5, 1.5){$A_3$}
	\put(6.0, 0.4){$B_1$}
	\put(7.0, 0.4){$B_2$}
	\put(8.0, 0.4){$B_3$}
	\put(4.4, 4.2){$C_1$}
	\put(4.4, 5.2){$C_2$}
	\put(4.4, 6.2){$C_3$}
	\put(4.4, 7.2){$C_4$}
	\put(4.4, 8.2){$C_5$}
	\color{blue}
	\linethickness{0.33mm}
	\path(4,3)(5,8)(6,1)(5,7)(3,2)(5,6)(7,1)(5,5)(2,2)(5,4)(8,1)(1,2)(4,3)
\end{picture}
}

\vsskip

This figure shows an example of such a polyline for $k=2$---its type is~\nck{12}{4}. This curious polyline (as well as many others presented in this article) belongs to the collection of Professor Knop's Zoo (see addendum of page~$\pageref{sec:knop_zoo}$). Unlike its more exotic sisters, this one is shown right here because it is but the first term of a sequence of similar polylines built using a specific and relatively elementary construction.

\vsskip

\fxkcb{11} Let us try to get some new examples of \nck nk polylines by a slight adjustment to procedure~\kcb{4}. Same as in \kcb{4}, consider natural number~$n$ and construct two regular polygons with $n$ sides, one of which is obtained from the other by homothety~$H$. This time, however, the homothety ratio is a negative number $-h$ for $h \gte 1$ (for instance, $h = 1.1$ or $h = 7.6$). Obviously, we need to make sure that no three of these $2n$ points lie on the same line. Further, choose two non-negative integers $p, q$ such that $\gcd(n, p+q) = 1$.

Index the vertices of the exterior polygon counterclockwise in the regular order, denoting them by $A_0$, $A_1$, ..., $A_{n-1}$ and treating the indices as residues modulo~$n$. Then denote the corresponding vertices of the interior polygon by $B_0$, $B_1$, ...,~$B_{n-1}$. In other words, the indexing is such that~$A_i = H(B_i)$.

Now, construct two families of segments: the segments in the first family connect vertices $A_i$ and $B_{i+p}$, while the segments of the second family connect vertices $A_i$ and~$B_{i-q}$. Collecting all the segments together, we obtain the closed polyline which we will name~\mdph{n}{p}{q}{h}. Since every vertex has exactly two edges issuing from it, the collection \mdph{n}{p}{q}{h} must be a union of several closed polylines. Clearly, after traversing two consecutive segments of this collection we return to the same polygon but with the index shift of $p+q$. Therefore, considering that $\gcd(n, p+q) = 1$, we conclude that any closed polyline in \mdph{n}{p}{q}{h} contains all vertices of each of the two given polygons. Hence, \mdph{n}{p}{q}{h} is indeed a closed polyline with $2n$ edges.

\vsskip

\centerline{\setlength{\unitlength}{7mm}
\begin{picture}(10,11)
  \put(8.12, 8.91){\circle*{0.23}}
  \put(2.92, 2.39){\circle*{0.23}}
  \put(3.89, 9.87){\circle*{0.23}}
  \put(5.74, 1.75){\circle*{0.23}}
  \put(0.50, 7.17){\circle*{0.23}}
  \put(8.00, 3.55){\circle*{0.23}}
  \put(0.50, 2.83){\circle*{0.23}}
  \put(8.00, 6.45){\circle*{0.23}}
  \put(3.89, 0.13){\circle*{0.23}}
  \put(5.74, 8.25){\circle*{0.23}}
  \put(8.12, 1.09){\circle*{0.23}}
  \put(2.92, 7.61){\circle*{0.23}}
  \put(10.00, 5.00){\circle*{0.23}}
  \put(1.67, 5.00){\circle*{0.23}}
  \put(0,10){\mdph{7}{1}{1}{1.5}}
  \color{blue}
  \linethickness{0.33mm}
  \path(8.12, 8.91)(5.74, 1.75)
  \path(3.89, 9.87)(8.00, 3.55)
  \path(0.50, 7.17)(8.00, 6.45)
  \path(0.50, 2.83)(5.74, 8.25)
  \path(3.89, 0.13)(2.92, 7.61)
  \path(8.12, 1.09)(1.67, 5.00)
  \path(10.00, 5.00)(2.92, 2.39)
  \path(8.12, 8.91)(1.67, 5.00)
  \path(3.89, 9.87)(2.92, 2.39)
  \path(0.50, 7.17)(5.74, 1.75)
  \path(0.50, 2.83)(8.00, 3.55)
  \path(3.89, 0.13)(8.00, 6.45)
  \path(8.12, 1.09)(5.74, 8.25)
  \path(10.00, 5.00)(2.92, 7.61)	
\end{picture}
\kern6ex
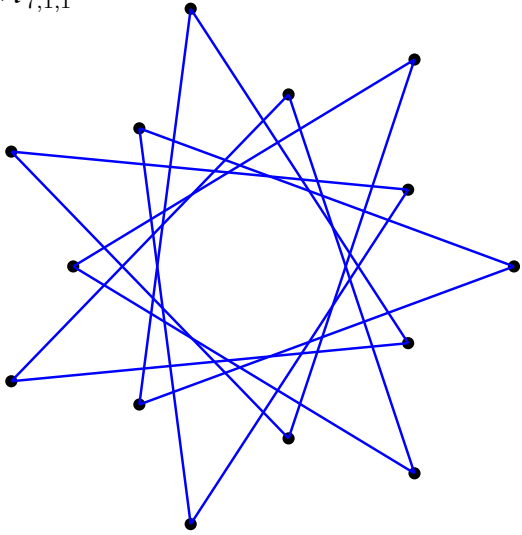
\begin{picture}(10,11)
  \put(8.54, 8.54){\circle*{0.23}}
  \put(2.28, 2.28){\circle*{0.23}}
  \put(5.00, 10.00){\circle*{0.23}}
  \put(5.00, 1.15){\circle*{0.23}}
  \put(1.46, 8.54){\circle*{0.23}}
  \put(7.72, 2.28){\circle*{0.23}}
  \put(0.00, 5.00){\circle*{0.23}}
  \put(8.85, 5.00){\circle*{0.23}}
  \put(1.46, 1.46){\circle*{0.23}}
  \put(7.72, 7.72){\circle*{0.23}}
  \put(5.00, 0.00){\circle*{0.23}}
  \put(5.00, 8.85){\circle*{0.23}}
  \put(8.54, 1.46){\circle*{0.23}}
  \put(2.28, 7.72){\circle*{0.23}}
  \put(10.00, 5.00){\circle*{0.23}}
  \put(1.15, 5.00){\circle*{0.23}}
  \put(0,10){\mdph{8}{1}{2}{1.3}}
  \color{blue}
  \linethickness{0.33mm}
  \path(8.54, 8.54)(5.00, 1.15)
  \path(1.46, 8.54)(8.85, 5.00)
  \path(0.00, 5.00)(7.72, 7.72)
  \path(1.46, 1.46)(5.00, 8.85)
  \path(5.00, 0.00)(2.28, 7.72)
  \path(8.54, 1.46)(1.15, 5.00)
  \path(10.00, 5.00)(2.28, 2.28)
  \path(8.54, 8.54)(2.28, 7.72)
  \path(5.00, 10.00)(1.15, 5.00)
  \path(1.46, 8.54)(2.28, 2.28)
  \path(0.00, 5.00)(5.00, 1.15)
  \path(1.46, 1.46)(7.72, 2.28)
  \path(5.00, 0.00)(8.85, 5.00)
  \path(8.54, 1.46)(7.72, 7.72)
  \color{red}
  \path(5.00, 10.00)(7.72, 2.28)
  \put(6.1,5.0){$7$}
  \color{darkgray}
  \path(10.00, 5.00)(5.00, 8.85)
  \put(8.6,6.2){$5$}
\end{picture}
}

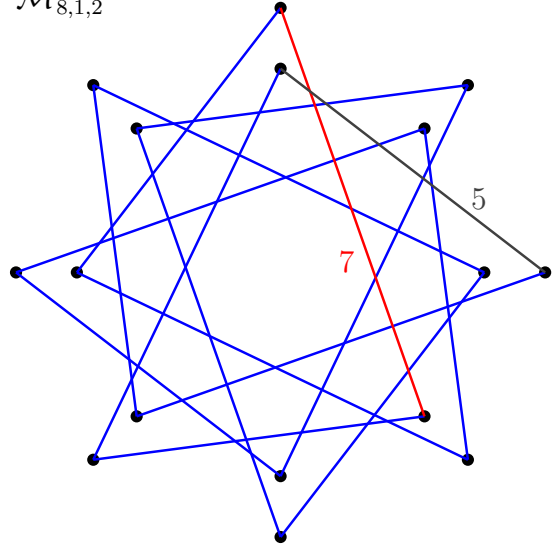
\vsskip

The figure above shows polylines \mdph{7}{1}{1}{1.5} and~\mdph{8}{1}{2}{1.3}.

\vsskip

\centerline{\setlength{\unitlength}{7mm}
\begin{picture}(10,11)
  \put(8.12, 8.91){\circle*{0.23}}
  \put(4.48, 4.35){\circle*{0.23}}
  \put(3.89, 9.87){\circle*{0.23}}
  \put(5.19, 4.19){\circle*{0.23}}
  \put(0.50, 7.17){\circle*{0.23}}
  \put(5.75, 4.64){\circle*{0.23}}
  \put(0.50, 2.83){\circle*{0.23}}
  \put(5.75, 5.36){\circle*{0.23}}
  \put(3.89, 0.13){\circle*{0.23}}
  \put(5.19, 5.81){\circle*{0.23}}
  \put(8.12, 1.09){\circle*{0.23}}
  \put(4.48, 5.65){\circle*{0.23}}
  \put(10.00, 5.00){\circle*{0.23}}
  \put(4.17, 5.00){\circle*{0.23}}
  \put(1,10){\ndph{7}{1}{6}}
  \color{blue}
  \linethickness{0.33mm}
  \path(8.12, 8.91)(5.19, 4.19)
  \path(3.89, 9.87)(5.75, 4.64)
  \path(0.50, 7.17)(5.75, 5.36)
  \path(0.50, 2.83)(5.19, 5.81)
  \path(3.89, 0.13)(4.48, 5.65)
  \path(8.12, 1.09)(4.17, 5.00)
  \path(10.00, 5.00)(4.48, 4.35)
  \path(8.12, 8.91)(4.17, 5.00)
  \path(3.89, 9.87)(4.48, 4.35)
  \path(0.50, 7.17)(5.19, 4.19)
  \path(0.50, 2.83)(5.75, 4.64)
  \path(3.89, 0.13)(5.75, 5.36)
  \path(8.12, 1.09)(5.19, 5.81)
  \path(10.00, 5.00)(4.48, 5.65)
\end{picture}
\kern6ex
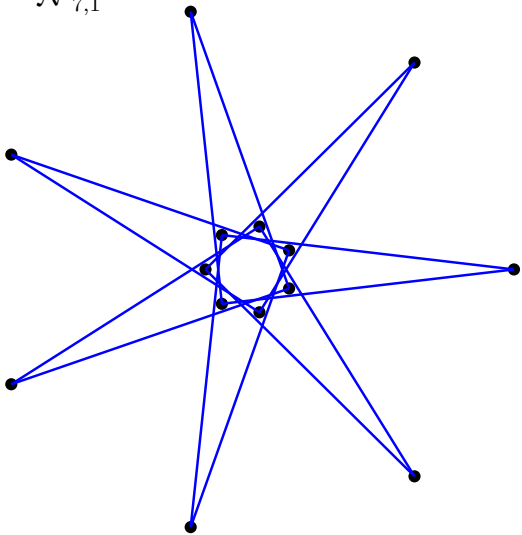
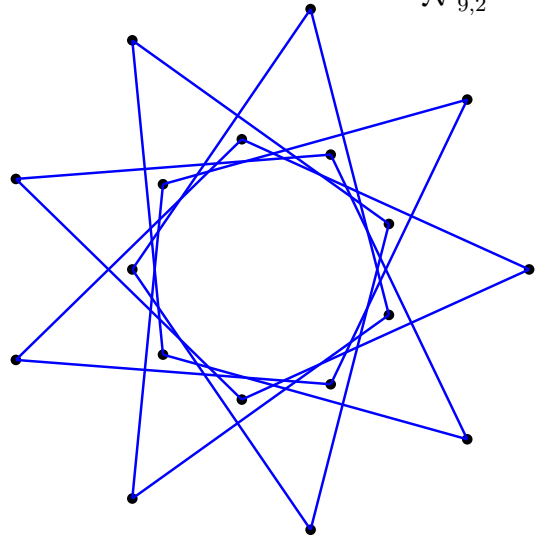
\begin{picture}(10,11)
  \put(8.83, 8.21){\circle*{0.20}}
  \put(3.08, 3.39){\circle*{0.20}}
  \put(5.87, 9.92){\circle*{0.20}}
  \put(4.57, 2.54){\circle*{0.20}}
  \put(2.50, 9.33){\circle*{0.20}}
  \put(6.25, 2.83){\circle*{0.20}}
  \put(0.30, 6.71){\circle*{0.20}}
  \put(7.35, 4.14){\circle*{0.20}}
  \put(0.30, 3.29){\circle*{0.20}}
  \put(7.35, 5.86){\circle*{0.20}}
  \put(2.50, 0.67){\circle*{0.20}}
  \put(6.25, 7.17){\circle*{0.20}}
  \put(5.87, 0.08){\circle*{0.20}}
  \put(4.57, 7.46){\circle*{0.20}}
  \put(8.83, 1.79){\circle*{0.20}}
  \put(3.08, 6.61){\circle*{0.20}}
  \put(10.00, 5.00){\circle*{0.20}}
  \put(2.50, 5.00){\circle*{0.20}}
  \put(8,10){\ndph{9}{2}{2}}
  \color{blue}
  \linethickness{0.33mm}
  \path(8.83, 8.21)(6.25, 2.83)
  \path(5.87, 9.92)(7.35, 4.14)
  \path(2.50, 9.33)(7.35, 5.86)
  \path(0.30, 6.71)(6.25, 7.17)
  \path(0.30, 3.29)(4.57, 7.46)
  \path(2.50, 0.67)(3.08, 6.61)
  \path(5.87, 0.08)(2.50, 5.00)
  \path(8.83, 1.79)(3.08, 3.39)
  \path(10.00, 5.00)(4.57, 2.54)
  \path(8.83, 8.21)(3.08, 6.61)
  \path(5.87, 9.92)(2.50, 5.00)
  \path(2.50, 9.33)(3.08, 3.39)
  \path(0.30, 6.71)(4.57, 2.54)
  \path(0.30, 3.29)(6.25, 2.83)
  \path(2.50, 0.67)(7.35, 4.14)
  \path(5.87, 0.08)(7.35, 5.86)
  \path(8.83, 1.79)(6.25, 7.17)
  \path(10.00, 5.00)(4.57, 7.46)
\end{picture}
}

\vsskip

It is easy to verify that in the first case, the self-intersection index for every edge is equal to~$8$, and therefore, the first polyline belongs to type~\nck{14}{8}. However, in the second polyline all the edges of the first family contain $7$ self-intersection points while that number for the second family equals~$5$. Generally speaking, even when $p \neq q$, these two self-intersection numbers could ``accidentally'' turn out to be equal if we were lucky with parameters $n$, $p$, $q$,~$h$---then polyline \mdph{n}{p}{q}{h} has the constant self-intersection index. That is, alas, a rather rare situation, and so we will concentrate exclusively on the case $p = q$, denoting polyline \mdph{n}{p}{p}{h} by~\ndph{n}{p}{h}.

Now, if $n$ is odd and such that $\gcd(n, p) = 1$, we will always (regardless of~$h$) obtain a closed polyline with fixed self-intersection index---in other words, \ndph{n}{p}{h} always has type \nck{2n}{k} for some value of parameter~$k$. It remains to compute the values that $k$ can take. For instance, when $n = 7$, $p = 1$, varying parameter $h$ yields us polylines of type \nck{14}{k} with $k = 6$ and~$k = 8$. The second figure above shows polylines \ndph{7}{1}{6} and~\ndph{9}{2}{2}. Quick check shows that these polylines satisfy conditions \nck{14}{6} and \nck{18}{6}, respectively.

\vsskip

\fxkcb{12} This final procedure generalizes many of polylines \pdph{n}{k}{h} constructed with~\kcb{4}.

Let us assume that we have a polyline of type \nck{n}{k} with odd~$n$. Double every edge of that polyline, similar to what we did in procedure \kcb{6}---but this time we will cross these doubled edges in every such pair.

It is easy to see that the resulting segment collection satisfies condition \nck{2n}{2k+1}, and, since $n$ is odd, these segments form one closed polyline. For instance, if we start with star~$\wstar(5,2)$, the result will be a polyline of type~\nck{10}{5} (see the figure below). In exactly the same manner, for any $k \gte 2$ we can transform a star of type $\wstar(2k+1,2)$ into a polyline of type~\nck{4k+2}{5}.

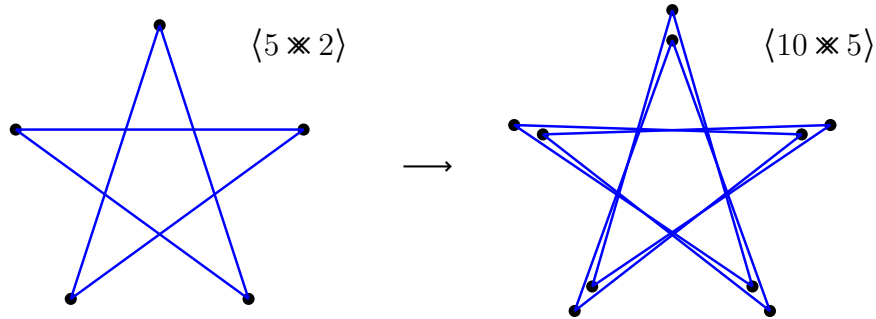
\vsskip

\centerline{\setlength{\unitlength}{4mm}
\begin{picture}(10,11)
  \put(0.24, 6.55){\circle*{0.40}}
  \put(2.06, 0.95){\circle*{0.40}}
  \put(7.94, 0.95){\circle*{0.40}}
  \put(9.76, 6.55){\circle*{0.40}}
  \put(5.00, 10.00){\circle*{0.40}}
  \put(8,9){\nck{5}{2}}
  \color{blue}
  \linethickness{0.33mm}
  \path(0.24, 6.55)(7.94, 0.95)
  \path(2.06, 0.95)(9.76, 6.55)
  \path(7.94, 0.95)(5.00, 10.00)
  \path(9.76, 6.55)(0.24, 6.55)
  \path(5.00, 10.00)(2.06, 0.95)
\end{picture}
\kern6ex
\begin{picture}(1,10)
  \put(0,5){$\longrightarrow$}
\end{picture}
\kern6ex
\begin{picture}(10,11)
  \put(-0.23, 6.70){\circle*{0.4}}
  \put(1.77, 0.55){\circle*{0.4}}
  \put(8.23, 0.55){\circle*{0.4}}
  \put(10.23, 6.70){\circle*{0.4}}
  \put(5.00, 10.50){\circle*{0.4}}
  \put(0.72, 6.39){\circle*{0.4}}
  \put(2.35, 1.36){\circle*{0.4}}
  \put(7.65, 1.36){\circle*{0.4}}
  \put(9.28, 6.39){\circle*{0.4}}
  \put(5.00, 9.50){\circle*{0.4}}
  \put(8,9){\nck{10}{5}}
  \color{blue}
  \linethickness{0.33mm}
  \path(-0.23, 6.70)(7.65, 1.36)
  \path(1.77, 0.55)(9.28, 6.39)
  \path(8.23, 0.55)(5.00, 9.50)
  \path(10.23, 6.70)(0.72, 6.39)
  \path(5.00, 10.50)(2.35, 1.36)
  \path(0.72, 6.39)(8.23, 0.55)
  \path(2.35, 1.36)(10.23, 6.70)
  \path(7.65, 1.36)(5.00, 10.50)
  \path(9.28, 6.39)(-0.23, 6.70)
  \path(5.00, 9.50)(1.77, 0.55)
\end{picture}
}

\vmskip

Hence, we have $\nck{n}{k} \implies \nck{2n}{2k+1}$ for odd $n$, or, equivalently,
$$
\nck{2p+1}{2q} \implies \nck{4p+2}{4q+1} \pp
$$
Here is another example. Star $\wstar(2p+1,p)$ is a polyline of type~\nck{2p+1}{2p-2}, and therefore, application of \kcb{12} proves that \nck{4p+2}{4p-3} is a feasible type for any~$p \gte 2$.

\section{\texorpdfstring{Computing $\mcc_3$}{Computing C-3}}
\hlabel{sec:c3}

For $k=1, 2$, we know the complete answer to the main question, and the solution in both cases is quite elementary (for example, it does not require any computer assistance). The case of $k=3$ is, arguably, the first non-trivial question in our investigation. At the same time, the setup of the problem is so simple and elementary that it can be presented even to fifth-grade students.

Since the self-intersection index is odd ($k=3$), the number of edges (that is, $n$) must be even. It is also obvious that $n \gte 8$ (case $n=6$ was already resolved in lemma~$\ref{thm:nk62})$. 

Polylines of types \nck{n}{3} for $n=10$, $12$, $14$ and $16$ can be found in Professor Knop's Zoo on page~$\pageref{fig:nck_z3}$. 
\hlabel{use:nck_z3}
Their pet names are $\mcz_{10,3}$, $\mcz_{12,3}$, $\mcz_{14,3}$, and~$\mcz_{16,3}$.

Applying procedure \kcb{6} for $p=6$ to a \nck{6}{1} polyline, we obtain a polyline of type~\nck{18}{3}. It can also be constructed applying procedure \kcb{3} to pair $n=9$,~$k=4$. Another, more intricate option: use procedure \kcb{9}, applying it to polyline~$\mcz_{12,3}$.

To construct polylines of types \nck{20}{3} and \nck{22}{3}, simply apply procedure \kcb{9} to the already constructed polylines \nck{14}{3} and~\nck{16}{3}. To do that, we just need to find in each of these polylines a suitable pair of edges. For polyline $\mcz_{14,3}$, pair $BA$, $DE$ is the one. Their intersection point is neighbor to vertices $A$ and $E$ respectively, and it is easy to verify that, if moving along the polyline in direction $B \to A$, we would visit $E$ before~$D$. For polyline $\mcz_{16,3}$, pair $KL$, $AQ$ does the trick.

Note that applying the same procedure to polyline $\mcz_{10,3}$ with its pair of edges $FG$, $JI$, we can build another example for~$n=16$. One more example of type \nck{14}{3} polyline is shown on page $\pageref{k4ex1}$ as an illustration of procedure~\kcb{4}.

For all even $n \gte 24$, the feasibility of type \nck{n}{3} is proved by applying procedure \kcb{8} to a \nck{n-12}{3} polyline, or by gluing together polylines of types \nck{n-10}{3} and~\nck{10}{3}.

The only type which presents a truly difficult challenge is~\nck{8}{3}. That is not a feasible type. However, at the moment, our only proof of that claim relies on the computer-assisted brute force analysis.

According to a communication from my colleague Konstantin Knop, he used a Python program, exploiting a MILP-style algorithm,\ftnnote{Mixed-Integer Linear Programming; this term refers to the problems in linear programming which involve both integer and floating types of numeric data.} to generate all possible $3315$ configurations of eight points in general position. Each of them produced $7!/2 = 2520$ polylines with none of them satisfying a \nck{8}{3} condition. Thus, the problem of computing set $\mcc_3$ is fully solved; it consists of all even numbers greater than~$8$.
$$
\mcc_3 = \{n \,|\, n\isdiv 2 \pc\ n \gte 10\} \pp
$$

\section{\texorpdfstring{Computing $\mcc_4$}{Computing C-4}}
\hlabel{sec:c4}

The case of $k=4$ is somewhat simpler (due to $k$ being even).

First, $n$ cannot be less than~$k+3 = 7$.

Second, for any natural $n\gte 7$ which is not a multiple of $3$, applying procedure \kcb{1} for $k=3$ yields a very simple and symmetric example of a type \nck{n}{4} polyline.

Third, when $n\equiv 0 \mmod 3$ and $n\gte 15$, number $m=n-8$ is not a multiple of $3$ and~$m\gte 7$. Therefore, a polyline of type \nck{m}{4} exists---and then application of \kcb{8} provides us with a polyline of type~\nck{n}{4}.

It follows that only two cases remain---namely, $n=9$ and~$n=12$. Polyline of type \nck{12}{4} has already been constructed as an example for procedure~\kcb{10}.

At this moment, the proof that type \nck{9}{4} is not feasible, alas, also requires computer assistance similar to the \nck{8}{3} case. The same MILP-based program, used by K.~Knop to investigate the former case, has generated $158817$ possible point configurations of nine points in general position. None of them could produce a \nck{9}{4} polyline.

Hence, set $\mcc_4$ consists of all natural numbers greater than $6$, save for~$9$.
$$
\mcc_4 = \{7,8\} \cup \{n \,|\, n \gte 10\} \pp
$$

\section{\texorpdfstring{Computing $\mcc_5$}{Computing C-5}}
\hlabel{sec:c5}

Clearly, when finding out whether type \nck{n}{5} is feasible, we only need to consider even numbers $n \gte 10$.

Procedure \kcb{12} allows us to construct a very simple polyline of type \nck{4m+2}{5} for any natural~$m \gte 2$. It remains to find examples of types \nck{4m}{5} for $m \gte 3$. Since any number $n = 4m$ for $m \gte 5$ can be represented as
$$
4m = 10 + \bigl(4(m-3)+2 \bigr) \pc \text{ for } m-3 \gte 2 \pc
$$
it follows that all these values of $n$, beginning with $20$, also belong to~$\mcc_5$. A \nck{16}{5} polyline {\hlabel{use:nck_z5}} can be found in the Zoo addendum on page~$\pageref{fig:nck_z5}$. However, we do not yet know if type \nck{12}{5} is feasible. Hence, set $\mcc_5$ is determined up to  this one member---with possible exclusion of number $12$ it consists of all even numbers greater than~$8$.

\section{\texorpdfstring{Computing $\mcc_6$}{Computing C-6}}
\hlabel{sec:c6}

For $k=6$, the situation is relatively simple but in some subcases it still presents a challenge.

First, note that for any odd $n \gte 9$ star $\wstar(n,4)$ gives us an example of type\nck{n}{6} polyline. Indeed, since $6 = 2 (4-1)$, procedure \kcb{1} can be applied to any value of $n$ which is greater than or equal to $2\cdot 4 + 1 = 9$ and co-prime with~$4$. Odd values of $n$ less than $9$ are certain not to belong to~$\mcc_6$.

Second, lemma \ref{thm:nnminus4} implies that type \nck{10}{6} is not feasible. Further, any even $n$, beginning with $18$, can be represented as the sum $9 + (n-9)$ (where the second summand $n-9$ is odd and is at least~$9$), and therefore, for such values of $n$ the desired example can be constructed by merging two polylines from the previous paragraph.

Further, a polyline of type \nck{16}{6} can be constructed by applying procedure \kcb{10} for~$k=3$. A type\nck{14}{6} polyline was already presented above as one of the examples for procedure \kcb{11}; namely, polyline~\ndph{7}{1}{6}. 

As for type \nck{12}{6}---a polyline for that type is shown in the left figure below.\ftnnote{This polyline is also a product of creative collaboration between Konstantin Knop and ~ChatGPT~$5.4$.} It is built on the vertices of the regular dodecagon indexed by residues modulo $12$ from $0$ to~$11$, which are then connected with each other using additive pattern~$(4,4,7)$. Namely, we connect the vertices in the following order:
\begin{align*}
a_0 = 0 \pc\ a_1 = a_0 + &\mathbf{4} = 4 \pc\ a_2 = a_1 + \mathbf{4} = 8 \pc\ a_3 = a_2 + \mathbf{7} = 3 \pc\ \\
a_4 = a_3 + &\mathbf{4} = 7 \pc\ a_5 = a_4 + \mathbf{4} = 11 \pc\ a_6 = a_5 + \mathbf{7} = 6 \pc\ \\
a_7 = a_6 + &\mathbf{4} = 10 \pc\ a_8 = a_7 + \mathbf{4} = 2 \pc\ a_9 = a_8 + \mathbf{7} = 9 \pc\ \\
a_10 = a_9 + &\mathbf{4} = 1 \pc\ a_{11} = a_{10} + \mathbf{4} = 5 \pc a_{12} = a_{11} + \mathbf{7} = 0 = a_0
\end{align*}
(all additions are done modulo~$12$).

Hence, we obtain that
$$
\mcc_6 = \{9\} \cup \{n \,|\, n \gte 11\} \pp
$$

\vsskip

\centerline{\setlength{\unitlength}{7mm}
\begin{picture}(10,11)
  \put(10.00, 5.00){\circle*{0.23}}
  \put(9.33, 7.50){\circle*{0.23}}
  \put(7.50, 9.33){\circle*{0.23}}
  \put(5.00, 10.00){\circle*{0.23}}
  \put(2.50, 9.33){\circle*{0.23}}
  \put(0.67, 7.50){\circle*{0.23}}
  \put(0.00, 5.00){\circle*{0.23}}
  \put(0.67, 2.50){\circle*{0.23}}
  \put(2.50, 0.67){\circle*{0.23}}
  \put(5.00, 0.00){\circle*{0.23}}
  \put(7.50, 0.67){\circle*{0.23}}
  \put(9.33, 2.50){\circle*{0.23}}
  \put(0,10){\nck{12}{6}}
  \color{blue}
  \linethickness{0.33mm}
  \path(10.00, 5.00)(0.67, 2.50)(9.33, 2.50)(5.00, 10.00)(7.50, 0.67)(7.50, 9.33
)(0.00, 5.00)(9.33, 7.50)(0.67, 7.50)(5.00, 0.00)(2.50, 9.33)(2.50, 0.67)(10.00,
 5.00)
\end{picture}
\kern6ex
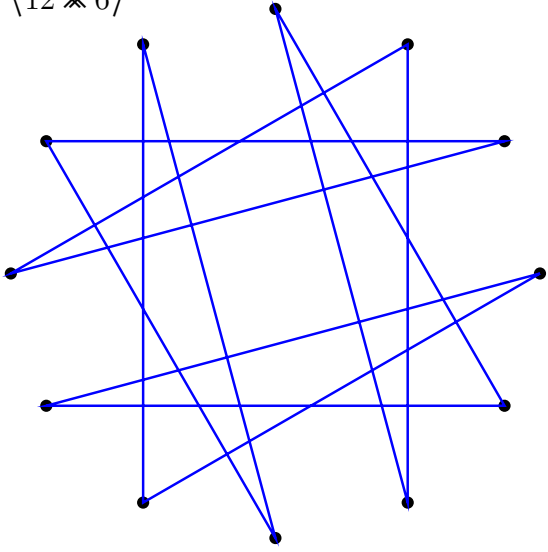
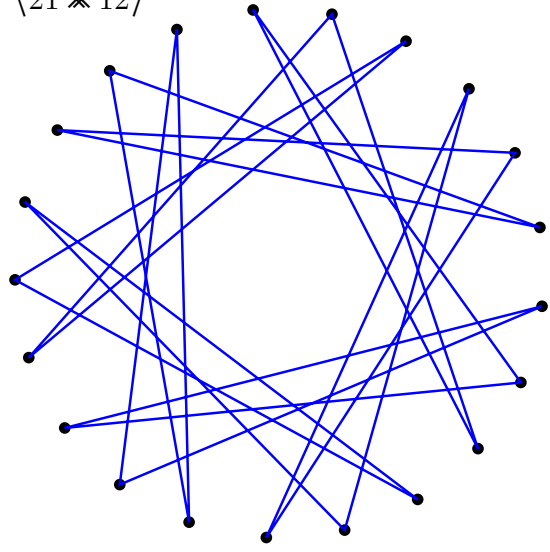
\begin{picture}(10,11)
  \put(9.92, 5.87){\circle*{0.22}}
  \put(1.79, 8.83){\circle*{0.22}}
  \put(3.29, 0.30){\circle*{0.22}}
  \put(3.06, 9.61){\circle*{0.22}}
  \put(1.98, 1.01){\circle*{0.22}}
  \put(9.96, 4.38){\circle*{0.22}}
  \put(0.94, 2.08){\circle*{0.22}}
  \put(9.56, 2.94){\circle*{0.22}}
  \put(4.50, 9.98){\circle*{0.22}}
  \put(8.75, 1.69){\circle*{0.22}}
  \put(5.99, 9.90){\circle*{0.22}}
  \put(0.26, 3.41){\circle*{0.22}}
  \put(7.39, 9.39){\circle*{0.22}}
  \put(0.00, 4.88){\circle*{0.22}}
  \put(7.61, 0.73){\circle*{0.22}}
  \put(0.19, 6.35){\circle*{0.22}}
  \put(6.23, 0.15){\circle*{0.22}}
  \put(8.58, 8.49){\circle*{0.22}}
  \put(4.75, 0.01){\circle*{0.22}}
  \put(9.45, 7.28){\circle*{0.22}}
  \put(0.80, 7.71){\circle*{0.22}}
  \put(0,10){\nck{21}{12}}
  \color{blue}
  \linethickness{0.33mm}
  \path(9.92, 5.87)(1.79, 8.83)(3.29, 0.30)(3.06, 9.61)(1.98, 1.01)(9.96, 4.38)(0.94, 2.08)(9.56, 2.94)(4.50, 9.98)(8.75, 1.69)(5.99, 9.90)(0.26, 3.41)(7.39, 9.39)(0.00, 4.88)(7.61, 0.73)(0.19, 6.35)(6.23, 0.15)(8.58, 8.49)(4.75, 0.01)(9.45, 7.28)(0.80, 7.71)(9.92, 5.87)
\end{picture}
}

\vmskip

\note. It is easy to see that construction of the \nck{12}{6} polyline above can be generalized to some larger values of $n$ and~$k$. On the right figure above, we show one such polyline of type~\nck{21}{12}. It was obtained by using regular polygon with $21$ vertices and additive pattern~$(7, 7, 13)$.

\section{\texorpdfstring{Computing $\mcb_n$: case $n=42$}{Computing B-42}}
\hlabel{sec:b42}

As a non-trivial example of computing $\mcb_n$, we will apply all the knowledge accumulated in the previous sections to the problem mentioned at the very beginning of the article---finding all elements of set~$\mcb_{42}$. In other words, we need to determine feasibility of types \nck{42}{k} for all natural values of~$k \lte 39$. We already know that \nck{42}{k} polylines exist when $k=1$ and $k=2$; that is, $\{1, 2\} \subset \mcb_{42}$. 

\vsskip

\noindent If $k=3$, then applying \kcb{3} for $n=21$, $k=4$ gives us
$$
n \gte 2k+1, \gcd(n,k) = 1 \overset{\strut \kcb{3}}{\implies} \nck{2n}{k-1} =  \nck{42}{3} \pp
$$
If $k=4$, apply \kcb{6} for $p=2$:
$$
21 \gte 8 \implies \nck{21}{2} \overset{\strut \kcb{6}}{\implies} \nck{42}{4} \pp
$$
If $k=5$, recall polyline \pdph{7}{2}{1.5} that we have built as an illustration for the generalized procedure \kcb{4} (see page~$\pageref{fig:k4ex2}$). That polyline has type \nck{14}{5}, and therefore, we can glue together three copies of it, obtaining a \nck{42}{5} example.
$$
\nck{14}{5} \overset{\kcb{5}}{\implies} \nck{42}{5}
$$
If $k=6$, use \kcb{6} for $p=3$:
$$
14 \gte 8 \implies \nck{14}{2} \overset{\strut \kcb{6}}{\implies} \nck{42}{6} \pp
$$
If $k=7$, apply \kcb{6} for $p=7$:
$$
6\isdiv 2, 6 \gte 6 \implies \nck{6}{1} \overset{\strut \kcb{6}}{\implies} \nck{42}{7} \pp
$$
If $k=8$, use \kcb{1}:
\begin{align*}
42 \gte 2\cdot 5 + 1, \gcd(42,5) = 1 &\overset{\strut \kcb{1}}{\implies} \nck{42}{8} \pp
\end{align*}
If $k=9$, use \kcb{3}:
$$
21 \gte 2\cdot 10 + 1, \gcd(21,10) = 1 \overset{\strut \kcb{3}}{\implies} \nck{42}{9} \pp
$$
If $k=10$, apply \kcb{1} for $n = 19$ and $n = 23$, and then use~\kcb{5}:
\begin{align*}
19 \gte 2\cdot 6 + 1, \gcd(19,6) = 1 &\overset{\strut \kcb{1}}{\implies} \nck{19}{10} \pc \\
23 \gte 2\cdot 6 + 1, \gcd(23,6) = 1 &\overset{\strut \kcb{1}}{\implies} \nck{23}{10} \pc \\
\nck{19}{10}, \nck{23}{10} &\overset{\strut \kcb{5}}{\implies} \nck{42}{10} \pp
\end{align*}

\noindent For $k=11$ we will again need the generalized procedure~\kcb{4}---namely, polyline~\pdph{21}{4}{1.3}. That polyline (shown in the figure below) has the desired type~\nck{42}{11}.

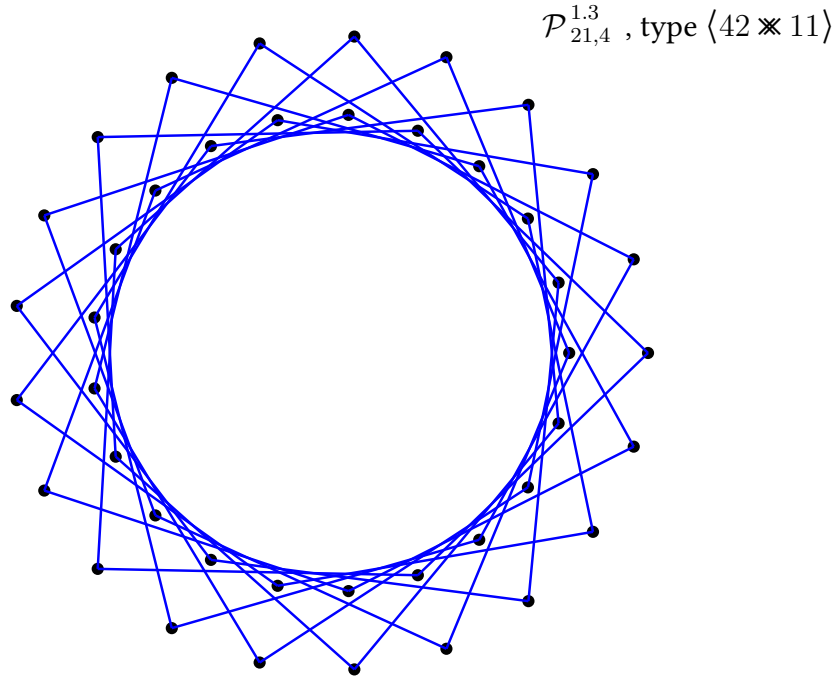
\vsskip

\centerline{\setlength{\unitlength}{7mm}
\begin{picture}(13,13)
  \put(11.73, 7.77){\circle*{0.23}}
  \put(10.31, 7.33){\circle*{0.23}}
  \put(10.96, 9.38){\circle*{0.23}}
  \put(9.73, 8.54){\circle*{0.23}}
  \put(9.74, 10.69){\circle*{0.23}}
  \put(8.81, 9.53){\circle*{0.23}}
  \put(8.19, 11.59){\circle*{0.23}}
  \put(7.65, 10.20){\circle*{0.23}}
  \put(6.45, 11.98){\circle*{0.23}}
  \put(6.34, 10.50){\circle*{0.23}}
  \put(4.66, 11.85){\circle*{0.23}}
  \put(5.00, 10.40){\circle*{0.23}}
  \put(3.00, 11.20){\circle*{0.23}}
  \put(3.74, 9.91){\circle*{0.23}}
  \put(1.60, 10.08){\circle*{0.23}}
  \put(2.69, 9.07){\circle*{0.23}}
  \put(0.59, 8.60){\circle*{0.23}}
  \put(1.94, 7.96){\circle*{0.23}}
  \put(0.07, 6.89){\circle*{0.23}}
  \put(1.54, 6.67){\circle*{0.23}}
  \put(0.07, 5.11){\circle*{0.23}}
  \put(1.54, 5.33){\circle*{0.23}}
  \put(0.59, 3.40){\circle*{0.23}}
  \put(1.94, 4.04){\circle*{0.23}}
  \put(1.60, 1.92){\circle*{0.23}}
  \put(2.69, 2.93){\circle*{0.23}}
  \put(3.00, 0.80){\circle*{0.23}}
  \put(3.74, 2.09){\circle*{0.23}}
  \put(4.66, 0.15){\circle*{0.23}}
  \put(5.00, 1.60){\circle*{0.23}}
  \put(6.45, 0.02){\circle*{0.23}}
  \put(6.34, 1.50){\circle*{0.23}}
  \put(8.19, 0.41){\circle*{0.23}}
  \put(7.65, 1.80){\circle*{0.23}}
  \put(9.74, 1.31){\circle*{0.23}}
  \put(8.81, 2.47){\circle*{0.23}}
  \put(10.96, 2.62){\circle*{0.23}}
  \put(9.73, 3.46){\circle*{0.23}}
  \put(11.73, 4.23){\circle*{0.23}}
  \put(10.31, 4.67){\circle*{0.23}}
  \put(12.00, 6.00){\circle*{0.23}}
  \put(10.51, 6.00){\circle*{0.23}}
  \put(10.0, 12.0){\pdph{21}{4}{1.3} \pc\ type \nck{42}{11}}
  \color{blue}
  \linethickness{0.33mm}
  \path(11.73, 7.77)(6.34, 10.50)
  \path(10.96, 9.38)(5.00, 10.40)
  \path(9.74, 10.69)(3.74, 9.91)
  \path(8.19, 11.59)(2.69, 9.07)
  \path(6.45, 11.98)(1.94, 7.96)
  \path(4.66, 11.85)(1.54, 6.67)
  \path(3.00, 11.20)(1.54, 5.33)
  \path(1.60, 10.08)(1.94, 4.04)
  \path(0.59, 8.60)(2.69, 2.93)
  \path(0.07, 6.89)(3.74, 2.09)
  \path(0.07, 5.11)(5.00, 1.60)
  \path(0.59, 3.40)(6.34, 1.50)
  \path(1.60, 1.92)(7.65, 1.80)
  \path(3.00, 0.80)(8.81, 2.47)
  \path(4.66, 0.15)(9.73, 3.46)
  \path(6.45, 0.02)(10.31, 4.67)
  \path(8.19, 0.41)(10.51, 6.00)
  \path(9.74, 1.31)(10.31, 7.33)
  \path(10.96, 2.62)(9.73, 8.54)
  \path(11.73, 4.23)(8.81, 9.53)
  \path(12.00, 6.00)(7.65, 10.20)
  \path(11.73, 7.77)(8.81, 2.47)
  \path(10.96, 9.38)(9.73, 3.46)
  \path(9.74, 10.69)(10.31, 4.67)
  \path(8.19, 11.59)(10.51, 6.00)
  \path(6.45, 11.98)(10.31, 7.33)
  \path(4.66, 11.85)(9.73, 8.54)
  \path(3.00, 11.20)(8.81, 9.53)
  \path(1.60, 10.08)(7.65, 10.20)
  \path(0.59, 8.60)(6.34, 10.50)
  \path(0.07, 6.89)(5.00, 10.40)
  \path(0.07, 5.11)(3.74, 9.91)
  \path(0.59, 3.40)(2.69, 9.07)
  \path(1.60, 1.92)(1.94, 7.96)
  \path(3.00, 0.80)(1.54, 6.67)
  \path(4.66, 0.15)(1.54, 5.33)
  \path(6.45, 0.02)(1.94, 4.04)
  \path(8.19, 0.41)(2.69, 2.93)
  \path(9.74, 1.31)(3.74, 2.09)
  \path(10.96, 2.62)(5.00, 1.60)
  \path(11.73, 4.23)(6.34, 1.50)
  \path(12.00, 6.00)(7.65, 1.80)
\end{picture}
}

\vsskip

\noindent If $k=12$, then use \kcb{6} for $p=6$:
$$
\nck{7}{2} \overset{\strut \kcb{6}}{\implies} \nck{42}{12} \pp
$$
Since $42$ is co-prime with numbers $11$, $13$, $17$, and $19$, procedure \kcb{1} generates polylines of type \nck{42}{k} for $k=20$, $24$, $32$,~$36$. 

Further, since $21$ is co-prime with $5$, $8$, and $10$, procedure \kcb{3} produces polylines of type \nck{21}{k} with $k = 8$, $14$,~$18$. Applying \kcb{5} allows us to double the number of edges and obtain polylines of type \nck{42}{k} for $k = 8$, $14$, $18$. Using procedure \kcb{6} for $p=2$ adds feasible values $k = 16$, $28$,~$36$.

Polylines \pdph{21}{4}{1.2}, \pdph{21}{5}{1.2}, \pdph{21}{8}{15}, \pdph{21}{8}{1.1}, and \pdph{21}{10}{2} serve as examples for $k=13$, $k=17$, $k=25$, $k=29$, and~$k = 37$. We are asking the readers to either believe our word of honor or to verify that claim by drawing those polylines on their own.

\vtskip

\noindent For $k=14$, use \kcb{1} to construct two polylines of types \nck{17}{14} and~\nck{25}{14}, and then glue them together.
$$
\nck{17}{14}, \nck{25}{14} \overset{\kcb{5}}{\implies} \nck{42}{14} \pp
$$
For $k=15$, apply procedure \kcb{6} for $p=3$ to polyline~\pdph{7}{2}{1.5} (which we have already met before):
$$
\nck{14}{5} \overset{\kcb{6}}{\implies} \nck{42}{15} \pp
$$
Similarly, we can do the same to solve the case of~$k=27$. Since we have in our possession polyline~\pdph{7}{3}{2.5} of type \nck{14}{9}, tripling both parameters $n$ and $k$ yields us an example of type~\nck{42}{27}.

Since $42$ is co-prime with numbers $11$, $13$, $17$, and $19$, using procedure \kcb{1} produces stars of types \nck{42}{k} for $k=20, 24, 32, 36$. 

Further, since $21$ is co-prime with $5$, $8$, and $10$, procedure \kcb{3} produces polylines of types \nck{21}{k} for $k = 8, 14, 18$. Then we can apply procedure \kcb{5} to double the number of edges and obtain polylines of types \nck{42}{k} for $k = 8, 14, 18$. Alternatively, applying procedure \kcb{6} with $p=2$ to the same polylines yields us examples for $k = 16, 28, 36$. 

Next, applying procedure \kcb{11}, we can construct several more examples for self-intersection indices $k=22, 30, 34$. Namely, polylines \ndph{21}{4}{1.5}, \ndph{21}{2}{3}, and \ndph{21}{1}{10} have types \nck{42}{22}, \nck{42}{30}, and \nck{42}{34}, respectively. The first of these polylines is shown below as an illustration.

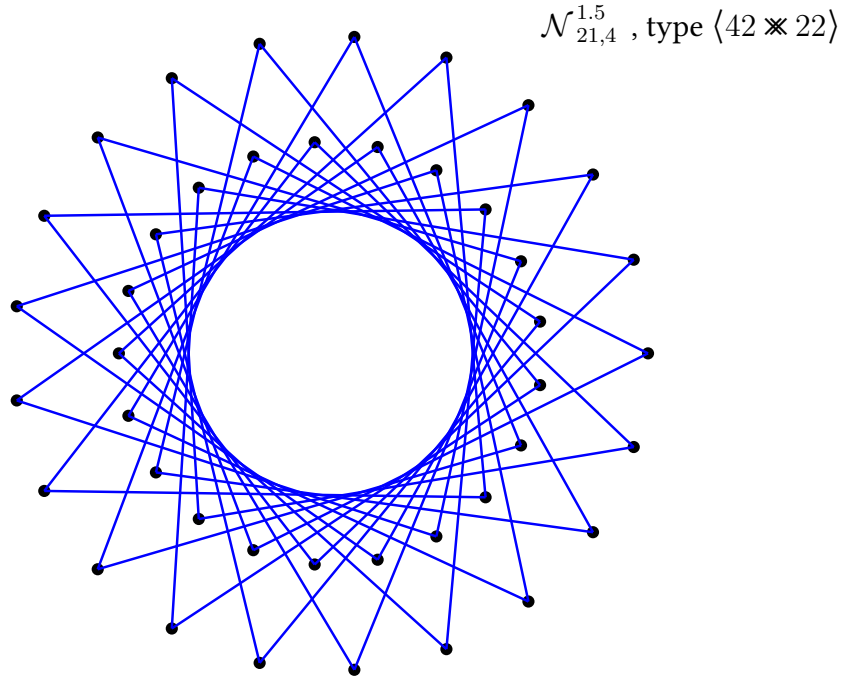
\vsskip

\centerline{\setlength{\unitlength}{7mm}
\begin{picture}(13,13)
  \put(11.73, 7.77){\circle*{0.23}}
  \put(2.18, 4.82){\circle*{0.23}}
  \put(10.96, 9.38){\circle*{0.23}}
  \put(2.70, 3.75){\circle*{0.23}}
  \put(9.74, 10.69){\circle*{0.23}}
  \put(3.51, 2.87){\circle*{0.23}}
  \put(8.19, 11.59){\circle*{0.23}}
  \put(4.54, 2.28){\circle*{0.23}}
  \put(6.45, 11.98){\circle*{0.23}}
  \put(5.70, 2.01){\circle*{0.23}}
  \put(4.66, 11.85){\circle*{0.23}}
  \put(6.89, 2.10){\circle*{0.23}}
  \put(3.00, 11.20){\circle*{0.23}}
  \put(8.00, 2.54){\circle*{0.23}}
  \put(1.60, 10.08){\circle*{0.23}}
  \put(8.93, 3.28){\circle*{0.23}}
  \put(0.59, 8.60){\circle*{0.23}}
  \put(9.60, 4.26){\circle*{0.23}}
  \put(0.07, 6.89){\circle*{0.23}}
  \put(9.96, 5.40){\circle*{0.23}}
  \put(0.07, 5.11){\circle*{0.23}}
  \put(9.96, 6.60){\circle*{0.23}}
  \put(0.59, 3.40){\circle*{0.23}}
  \put(9.60, 7.74){\circle*{0.23}}
  \put(1.60, 1.92){\circle*{0.23}}
  \put(8.93, 8.72){\circle*{0.23}}
  \put(3.00, 0.80){\circle*{0.23}}
  \put(8.00, 9.46){\circle*{0.23}}
  \put(4.66, 0.15){\circle*{0.23}}
  \put(6.89, 9.90){\circle*{0.23}}
  \put(6.45, 0.02){\circle*{0.23}}
  \put(5.70, 9.99){\circle*{0.23}}
  \put(8.19, 0.41){\circle*{0.23}}
  \put(4.54, 9.72){\circle*{0.23}}
  \put(9.74, 1.31){\circle*{0.23}}
  \put(3.51, 9.13){\circle*{0.23}}
  \put(10.96, 2.62){\circle*{0.23}}
  \put(2.70, 8.25){\circle*{0.23}}
  \put(11.73, 4.23){\circle*{0.23}}
  \put(2.18, 7.18){\circle*{0.23}}
  \put(12.00, 6.00){\circle*{0.23}}
  \put(2.00, 6.00){\circle*{0.23}}
  \put(10.0, 12.0){\ndph{21}{4}{1.5} \pc\ type \nck{42}{22}}
  \color{blue}
  \linethickness{0.33mm}
  \path(11.73, 7.77)(5.70, 2.01)
  \path(10.96, 9.38)(6.89, 2.10)
  \path(9.74, 10.69)(8.00, 2.54)
  \path(8.19, 11.59)(8.93, 3.28)
  \path(6.45, 11.98)(9.60, 4.26)
  \path(4.66, 11.85)(9.96, 5.40)
  \path(3.00, 11.20)(9.96, 6.60)
  \path(1.60, 10.08)(9.60, 7.74)
  \path(0.59, 8.60)(8.93, 8.72)
  \path(0.07, 6.89)(8.00, 9.46)
  \path(0.07, 5.11)(6.89, 9.90)
  \path(0.59, 3.40)(5.70, 9.99)
  \path(1.60, 1.92)(4.54, 9.72)
  \path(3.00, 0.80)(3.51, 9.13)
  \path(4.66, 0.15)(2.70, 8.25)
  \path(6.45, 0.02)(2.18, 7.18)
  \path(8.19, 0.41)(2.00, 6.00)
  \path(9.74, 1.31)(2.18, 4.82)
  \path(10.96, 2.62)(2.70, 3.75)
  \path(11.73, 4.23)(3.51, 2.87)
  \path(12.00, 6.00)(4.54, 2.28)
  \path(11.73, 7.77)(3.51, 9.13)
  \path(10.96, 9.38)(2.70, 8.25)
  \path(9.74, 10.69)(2.18, 7.18)
  \path(8.19, 11.59)(2.00, 6.00)
  \path(6.45, 11.98)(2.18, 4.82)
  \path(4.66, 11.85)(2.70, 3.75)
  \path(3.00, 11.20)(3.51, 2.87)
  \path(1.60, 10.08)(4.54, 2.28)
  \path(0.59, 8.60)(5.70, 2.01)
  \path(0.07, 6.89)(6.89, 2.10)
  \path(0.07, 5.11)(8.00, 2.54)
  \path(0.59, 3.40)(8.93, 3.28)
  \path(1.60, 1.92)(9.60, 4.26)
  \path(3.00, 0.80)(9.96, 5.40)
  \path(4.66, 0.15)(9.96, 6.60)
  \path(6.45, 0.02)(9.60, 7.74)
  \path(8.19, 0.41)(8.93, 8.72)
  \path(9.74, 1.31)(8.00, 9.46)
  \path(10.96, 2.62)(6.89, 9.90)
  \path(11.73, 4.23)(5.70, 9.99)
  \path(12.00, 6.00)(4.54, 9.72)
\end{picture}
}

\vsskip

\hlabel{use:nck_z7}
For $k = 21$ we will use polyline $\mcz_{14,7}$ (see page~$\pageref{fig:nck_z7}$), which we will ``triple'' by using procedure~\kcb{6}.
$$
\nck{14}{7} \overset{\kcb{6}}{\implies} \nck{42}{21} \pp
$$

\hlabel{use:nck_42z}
Finally, we claim the existence of the examples for $k=19, 23, 26, 31, 33, 35$---they live in the same remote zoo (see addendum at the end of the article). However, just one of them, $\mcz_{42,19}$ (page~$\pageref{fig:nck_42z}$) is actually presented there. We think that the reader will only need one peek at that illustration to understand our decision not to show you the other animals from the same complicated and esoteric family.

Due to lemmas $\ref{thm:nnminus3}$ and $\ref{thm:nnminus4}$ we also know that $38, 39 \notin \mcb_{42}$. Hence, set $\mcb_{42}$ is the interval consisting of all integer numbers from $1$ to~$37$.

\section{\texorpdfstring{When $n$ is sufficiently large. Part I}{When n is sufficiently large. Part I}}

\begin{theorem}
\hlabel{thm:nk_lower_bound_1}
Let $n$ and $k$ be two natural numbers such that at least one of them is even. Further, let $C = 3$ if $k$ is even, and $C = 20$ if $k$ is odd. Then a polyline of type \nck{n}{k} exists for any $n \gte C(k+1)$.
\end{theorem}

\begin{proof}
To prove it, we will use the following simple fact from number theory.

\begin{lemma}
\hlabel{thm:ab_sum_coprime_12}
Let $s$ and $t$ be two natural numbers. Then

(a) if $t$ is odd and $s \gte t$, then $s$ can be represented as a sum of no more than two natural numbers co-prime with~$t$;

(b) if $t$ is even and $s \gte 2t$, then $s$ can be represented as a sum of no more than three natural numbers co-prime with~$t$.
\end{lemma}

\begin{mproof}
Express $t$ as the product of powers of primes:
$$
t = p_1^{\alpha_1} ... p_m^{\alpha_m} \pc\ p_1 < \ldots < p_m \pp
$$
Now we will prove that any residue modulo $t$ can be represented as the sum of $d$ invertible residues, where $d$ is either two or three.

First, assume that $t$ is odd, and prove this for $d=2$. By the \href{https://en.wikipedia.org/wiki/Chinese_remainder_theorem}{Chinese remainder theorem}, it suffices to prove this for the case when $t$ is a power of an odd prime~$p$.

Indeed, consider any residue $x \in \Z/t\Z$ and express it as
$$
x = 1 + (x-1) \pc\ x = 2 + (x-2) \pp
$$
Residues $1$ and $2$ are invertible in $\Z/t\Z$, and at least one of the elements $x-1$ and $x-2$ of that ring must be invertible in $\Z/t\Z$, since at least one of these numbers is co-prime with~$p \gte 3$.

The remaining case is when $t$ is even. Then any element $x \in \Z/t\Z$ can be expressed as the sum of two or three invertible elements. Just as before, it is enough to prove this for any power of a prime. For a power of two we have
$$
x = 1 + (x-1) = 1 + 1 + (x-2) \pp
$$
One of the numbers $x-1$ and $x-2$ is odd and therefore is invertible in~$\Z/2^\alpha\Z$. It remains to prove that for any power of an odd prime $p$ any residue can be represented both as the sum of two and as the sum of three invertible elements. Indeed,
$$
x = 1 + (x-1) = 2 + (x-2) \pc\ \  x = 1 + 1 + (x-2) = 1 + 2 + (x-3) \pp
$$
Hence, if $x-2$ is invertible, then we have $x = 2 + (x-2) = 1 + 1 + (x-2)$. And if $x-2$ is not co-prime with $p$, then we will use representation $x = 1 + (x-1) = 1 + 2 + (x-3)$.

Consider now number $s$ and its remainder modulo~$t$. If $t$ is odd, that remainder can be represented as the sum of two invertible residues modulo~$t$. Let us regard these residues as ordinary natural numbers not exceeding~$t$. Since $s \gte t$, the sum of these two numbers cannot exceed~$s$. Therefore, by adding multiples of $t$ if necessary, we will obtain that $s$ equals the sum of two natural numbers co-prime with~$t$, and we are done.

The proof for the case of even $t$ and representation of $s \mmod t$ as the sum of no more than three invertible residues is entirely analogous.
\end{mproof}

\begin{corollary}
\hlabel{thm:ab_sum_coprime_35}
(a) If $t$ is odd and $s \gte 3t$, then $s$ can be represented as a sum of no more than two natural numbers greater than or equal to $t$ and co-prime with it.

(b) If $t$ is even and $s \gte 5t$, then $s$ can be represented as a sum of no more than three natural numbers greater than or equal to $t$ and co-prime with it.
\end{corollary}

\begin{mproof}
Let us prove item (b) (the proof of (a) is almost identical). We will use lemma $\ref{thm:ab_sum_coprime_12}$(b) to find some representation of $s$ as a sum of two or three natural numbers co-prime with~$t$. If one of these numbers---call it $x$---is less than $t$, then the sum of the remaining summands is greater than $4t$, and therefore, one of them---say, $y$---is greater than~$2t$. We can now replace numbers $x$, $y$ with $x+t$, $y-t$, obtaining two numbers with the same sum which are both greater than or equal to~$t$. If there is a third summand which is less than $t$, we can perform the same operation one more time.
\end{mproof}

Now back to the proof of the theorem. We may assume that $k \gte 2$ (case $k=1$ is well known, and we have already constructed the necessary polylines for all even values of $n \gte 6$, which is more than enough in this particular case).

Now, consider the case when $k$ is odd and~$n = 2m$. Since $n \gte 20(k+1)$, we can apply item (b) of the corollary $\ref{thm:ab_sum_coprime_35}$ to the pair $s = m$, $t = 2(k+1)$. This yields the representation $m = x_1 + x_2 + \ldots  + x_s$, where every $x_i$ is greater than or equal to $t+1 = 2k+3$ and is co-prime with~$k+1$. Hence, we can apply procedure \kcb{3} for the pair $(x_i,k)$, obtaining a polyline of type~\nck{2x_i}{k}. Now it suffices to apply procedures \kcb{5}, gluing all these polylines together. Since $2(x_1 + \ldots + x_s) = 2m = n$, the resulting polyline will have type~\nck{n}{k}. 

If $k$ is even, then item (a) of corollary $\ref{thm:ab_sum_coprime_35}$ can be applied to the pair $s = n$, $t = k+1$. In order to do that, we only need to make sure that $n \gte 3(k+1)$. Now we can use procedures \kcb{1} for every pair $(x_i,k)$, and then glue together all the polylines we will have obtained.
\end{proof}

\begin{corollary}
Let $k$ be a fixed natural number. Then any sufficiently large natural number $n$ belongs to~$\mcc_k$, provided that $nk$ is even.
\end{corollary}

\section{\texorpdfstring{When $n$ is sufficiently large. Part II}{When n is sufficiently large. Part II}}

Another idea that allows us to obtain a more accurate lower bound for the elements of set~$\mcc_k$, consists of using a version of polyline surgery which generalizes the one described in procedures \kcb{5} and~\kcb{6} (this approach was sketched out by Alexander Kovalji in his article~\cite{akk_2003}; here we give a more or less complete description of that method, including small improvement of the bound given there).

This approach, obviously, allows us to improve the lower bound for the case of odd $k$ as well.

\vtskip

\begin{theorem}
\hlabel{thm:nk_lower_bound_2}
Number $n$ belongs to $\mcc_k$ (in other words, polylines of type \nck nk exist), if one of the following two conditions is met:

(a) Number $k$ is even and $n \gte 2k+3$;

(b) Number $k$ is odd and $n$ is an even number such that~$n \gte 8k+6$.
\end{theorem}

\begin{proof}
Let us begin with the case when $k$ is even and equal to~$2p$. From Theorem \ref{thm:nk_lower_bound_1}(a), it follows that we only need to prove the stated claim for $2k+3 \lte n \lte 3k+2$.

For values $n = 2k+3$ and $n = 2k+5$, it suffices to apply procedure~\kcb{1}. Indeed, 
$$
n = 4p + 4 \pm 1 \gte 2p+3 \pc\ \gcd(4(p+1)\pm 1, p+1) = 1 \pc
$$
and therefore, there exists a polyline of type~\nck{n}{2p}. Further, procedure \kcb{10} yields an example of a type $\nck{4p+4}{2p}$ polyline, which proves the claim for~$n=2k+4$.

Hence, we may assume that $n \gte 2k+6$. As we already know, we can construct a very simple polyline $\mcl_1$ of type \nck{k+3}{k}, formed by the main diagonals of an arbitrary convex polygon $M_1$ with $k+3$ vertices. Consider now type \nck {k+d}{k}, where $d$ is some natural number from $3$ to~$k-1$. Take the vertices of some other convex polygon $M_2$ with $k+d$ vertices, and draw all diagonals of span~$p+1$. In this collection of segments---we will call it $\mcl_2$---each of its segments crosses exactly $2p = k$ other segments belonging to the collection. If $\mcl_2$ were a closed polyline, we could then use procedure \kcb{5}, and, by gluing together $\mcl_1$ and $\mcl_2$, obtain a polyline of type \nck{2k+d+3}{k} for any $d$ from $3$ to~$k-1$. That would give us all types \nck{n}{k} for $2k+6 \lte n \lte 3k+2$. However, in general, $\mcl_2$ is a union of $s$ separate closed polylines where
$$
s = \gcd(2p+d, p+1) = \gcd(d-2, p+1) \pp
$$
Hence, $s$ can easily be greater than $1$, but it is obvious that $s \lte p+1$. At the same time each of the polygons $\mcl_1$ and $\mcl_2$ has at least $2p+3$ vertices. Therefore, the number of vertices in each polygon is more than twice the number of components of collection~$\mcl_2$. It is also obvious that if we index the vertices of $M_2$ along the polygon's contour, then any $s$ consecutive vertices belong to different components.

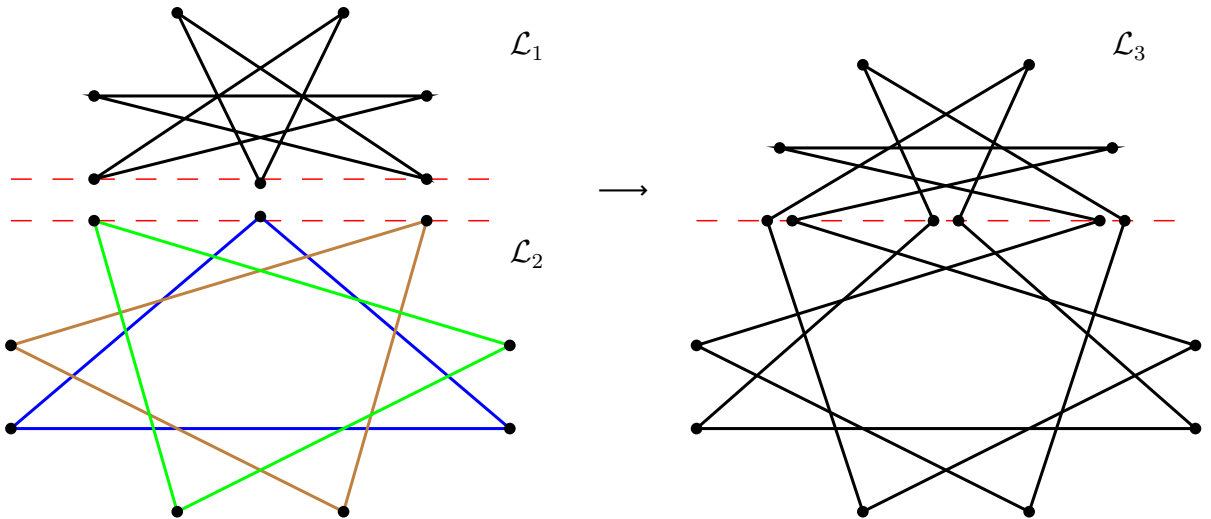
\vmskip

\centerline{\setlength{\unitlength}{5.5mm}
\begin{picture}(13,12)
	\linethickness{0.2mm}
	\color{red}
	\multiput(0.0,7.0)(1,0){12}{\path(0,0)(0.5,0)}
	\multiput(0.0,8.0)(1,0){12}{\path(0,0)(0.5,0)}
    \linethickness{0.4mm}
    \color{blue}
	\polyline(0.0,2.0)(6.0,7.1)(12.0,2.0)(0.0,2.0)
    \color{brown}
	\polyline(0.0,4.0)(8.0,0.0)(10.0,7.0)(0.0,4.0)
    \color{green}
	\polyline(4.0,0.0)(2.0,7.0)(12.0,4.0)(4.0,0.0)
	%
    \color{black}
	\put(0.0,2.0){\circle*{0.28}}
	\put(0.0,4.0){\circle*{0.28}}
	\put(2.0,7.0){\circle*{0.28}}
	\put(4.0,0.0){\circle*{0.28}}
	\put(6.0,7.1){\circle*{0.28}}
	\put(8.0,0.0){\circle*{0.28}}
	\put(10.0,7.0){\circle*{0.28}}
	\put(12.0,2.0){\circle*{0.28}}
	\put(12.0,4.0){\circle*{0.28}}
	\put(12.0,6.0){$\mcl_2$}
	%
    \color{black}
	\put(2.0,8.0){\circle*{0.28}}
	\put(2.0,10.0){\circle*{0.28}}
	\put(10.0,8.0){\circle*{0.28}}
	\put(10.0,10.0){\circle*{0.28}}
	\put(6.0,7.9){\circle*{0.28}}
	\put(4.0,12.0){\circle*{0.28}}
	\put(8.0,12.0){\circle*{0.28}}
	\put(12.0,11.0){$\mcl_1$}
	\polyline(2.0,8.0)(10.0,10.0)(2.0,10.0)(10.0,8.0)(4.0,12.0)(6.0,7.9)(8.0,12.0)(2.0,8.0)
\end{picture}
\kern1ex
\begin{picture}(2,12)
	\put(0.5,7.5){$\longrightarrow$}
\end{picture}
\kern2ex
\begin{picture}(12,12)
	\linethickness{0.2mm}
	\color{red}
	\multiput(0.0,7.0)(1,0){12}{\path(0,0)(0.5,0)}
    \linethickness{0.4mm}
    \color{black}
	\put(0.0,2.0){\circle*{0.28}}
	\put(0.0,4.0){\circle*{0.28}}
	\put(1.7,7.0){\circle*{0.28}}
	\put(2.3,7.0){\circle*{0.28}}
	\put(4.0,0.0){\circle*{0.28}}
	\put(5.7,7.0){\circle*{0.28}}
	\put(6.3,7.0){\circle*{0.28}}
	\put(8.0,0.0){\circle*{0.28}}
	\put(9.7,7.0){\circle*{0.28}}
	\put(10.3,7.0){\circle*{0.28}}
	\put(12.0,2.0){\circle*{0.28}}
	\put(12.0,4.0){\circle*{0.28}}
	\put(2.0,8.75){\circle*{0.28}}
	\put(10.0,8.75){\circle*{0.28}}
	\put(4.0,10.75){\circle*{0.28}}
	\put(8.0,10.75){\circle*{0.28}}
	\put(10.0,11.0){$\mcl_3$}
	\polyline(5.7,7.0)(0.0,2.0)(12.0,2.0)(6.3,7.0)
	\polyline(10.3,7.0)(8.0,0.0)(0.0,4.0)(9.7,7.0)
	\polyline(2.3,7.0)(12.0,4.0)(4.0,0.0)(1.7,7.0)
	\polyline(1.7,7.0)(8.0,10.75)(6.3,7.0)
	\polyline(2.3,7.0)(10.0,8.75)(2.0,8.75)(9.7,7.0)
	\polyline(5.7,7.0)(4.0,10.75)(10.3,7.0)
\end{picture}%
}

\vmskip

This implies that we can construct and position polygons $M_1$ and $M_2$ as shown in the figure above---so that $s$ consecutive ``lower'' vertices of $M_1$ have nearly the same $y$-coordinate, and the same is true for some $s$ consecutive ``upper'' vertices of~$M_2$. Now, shift the upper polygon $M_1$ downward so that the corresponding ``corners'' of $M_1$ and $M_2$ will overlap; then in the neighborhood of each of the $s$ overlaps we can apply the surgery already demonstrated in procedures \kcb{5} and~\kcb{6}. After this $s$-fold surgery, all $s+1$ original closed polylines in $\mcl_1$ and $\mcl_2$ will merge into one closed polyline of type \nck{2k+d+3}{k}, completing the construction.

In the figure above, this procedure is performed for $k = 4$, $p = 2$, $d = 5$, $s = \gcd(d-2,p+1) = 3$. Hence, we are gluing together a $7$-edge polyline $\mcl_1$ of type \nck{7}{4} and collection $\mcl_2$ of three closed polylines, which satisfies the \nck{9}{4} condition, resulting in the closed polyline $\mcl_3$ of type~\nck{16}{4}.

Finally, item (b) can be easily derived from~(a). Indeed, use item (a), applied to~$k^* = 2k$. Then for any $m \gte 2k^*+3 = 4k+3$ there exists a polyline of type $\nck{m}{k^*} = \nck{m}{2k}$. Using procedure \kcb{2} for $p=2$ yields us a polyline of type~\nck{2m}{k}, finalizing the proof of our theorem.
\end{proof}

\section{Unsolved problems and conjectures}

This section is a simple compilation of a few unsolved problems, including one conjecture.

\begin{qnxbox}{Non-feasibility for \nck{8}{3} and~\nck{9}{4}}
\noindent
Find proofs of non-feasibility for the types \nck{8}{3} and~\nck{9}{4} that do not rely on computer assistance.
\end{qnxbox}

\begin{qnxbox}{Computing $\mcc_5$}
\noindent 
Is \nck{12}{5} a feasible type?
\end{qnxbox}

\begin{qnxbox}{Type \nck{2k+2}{2}}
\noindent 
Is is true that for any $k \gte 4$ there exists a polyline of type \nck{2k+2}{k}\,?
\end{qnxbox}

\begin{qnxbox}{\textbf{Conjecture}: if $nk\isdiv 2$, then type \nck{n}{k} is feasible for any $n \gte 2k+4$}
\noindent 
Provided that at least one of the numbers $n$, $k$ is even and $n \gte 2k+4$, does there always exist a polyline of type \nck{n}{k}\,?
\end{qnxbox}

\vtskip

Finally, the general problem, whose solution is not likely to be found in the next few years.

\begin{qnxbox}{Polylines with fixed self-intersection index}
\noindent 
Describe all pairs of natural numbers $n$ and $k$ for which type \nck{n}{k} is feasible.
\end{qnxbox}

\section*{Acknowledgments}

The author would like to express his gratitude to Alexey Vyskubov and Alexander Merkurjev for fruitful discussions on this topic. 

Special thanks go to Konstantin Knop for numerous examples of closed polylines with constant self-intersection index that he constructed in close and friendly collaboration with AI model ChatGPT~$5.4$ (\copyright~OpenAI,~$2026$); in particular, see addendum ``Professor Knop's Zoo'' on page~$\pageref{sec:knop_zoo}$.

For help with generating the snippets of code used in some drawings and for constructing computational models for the search of \nck{n}{k} polylines, the author would like to thank everyone who contributed to the creation and training of the AI engine Claude Sonnet~$4.5$ (\copyright~Anthropic,~$2025$).


\vmskip

\clearpage 
\section*{Addendum: Professor Knop's Zoo}
\hlabel{sec:knop_zoo}

All the polylines in this section were graciously communicated to the author by his friend and colleague Konstantin Knop. They are the fruit of a very successful collaboration between him and his rather capricious research assistant ChatGPT~$5.4$. We are referring to some of these polylines (under their proper names of the form $\mcz_{12,3}$ or $\mcz_{14,7}$) in the main body of this article.

As it turned out, some of these examples can be generated using procedures \kcb{4} or \kcb{12}, but we decided to keep them here so that our readers could enjoy the sight of these exotic animals.

We begin with the series of polylines with self-intersection index equal to~$3$. These examples were used in section ``Computing $\mcc_3$'' on page~$\pageref{use:nck_z3}$.

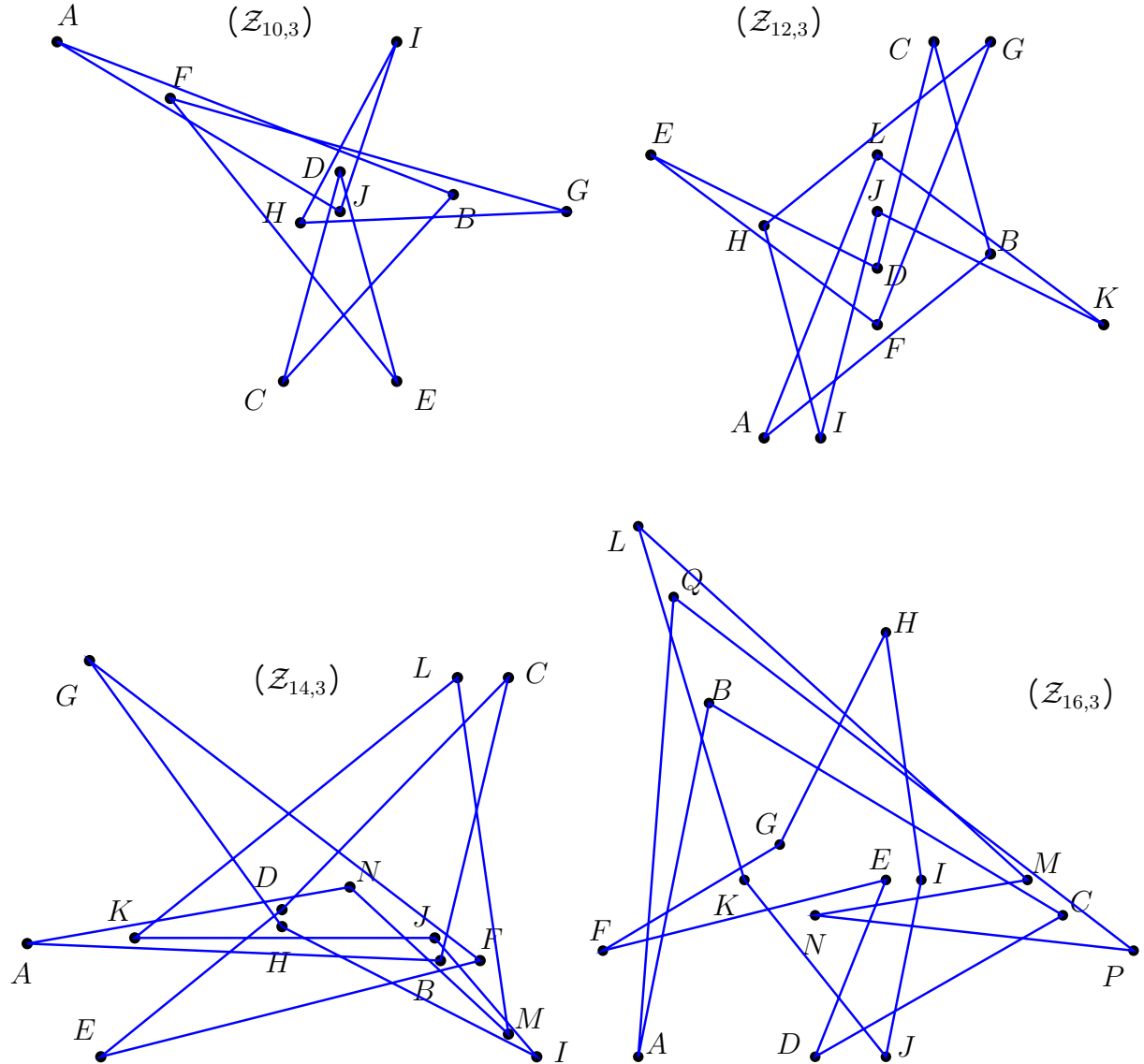
\vmskip
\hlabel{fig:nck_z3}

\centerline{\setlength{\unitlength}{8mm}
\begin{picture}(9,8)
	\put(0.0,7.0){\circle*{0.2}}  
	\put(7.0,4.3){\circle*{0.2}}  
	\put(4.0,1.0){\circle*{0.2}}  
	\put(5.0,4.7){\circle*{0.2}}  
	\put(6.0,1.0){\circle*{0.2}}  
	\put(2.0,6.0){\circle*{0.2}}  
	\put(9.0,4.0){\circle*{0.2}}  
	\put(4.3,3.8){\circle*{0.2}}  
	\put(6.0,7.0){\circle*{0.2}}  
	\put(5.0,4.0){\circle*{0.2}}  
	\put(3.0,7.2){$(\mcz_{10,3})$}
	\put(0.0,7.3){$A$}
	\put(7.0,3.7){$B$}
	\put(3.3,0.5){$C$}
	\put(4.3,4.6){$D$}
	\put(6.3,0.5){$E$}
	\put(2.0,6.2){$F$}
	\put(9.0,4.2){$G$}
	\put(3.6,3.8){$H$}
	\put(6.2,6.9){$I$}
	\put(5.2,4.1){$J$}
	\color{blue}
	\linethickness{0.33mm}
	\path(0.0,7.0)(7.0,4.3)(4.0,1.0)(5.0,4.7)(6.0,1.0)(2.0,6.0)(9.0,4.0)(4.3,3.8)(6.0,7.0)(5.0,4.0)(0.0,7.0)
\end{picture}
\kern6ex
\begin{picture}(8,8)
	\put(2,0){\circle*{0.2}}
	\put(6,3.25){\circle*{0.2}}
	\put(5,7){\circle*{0.2}}
	\put(4,3){\circle*{0.2}}
	\put(0,5){\circle*{0.2}}
	\put(4,2){\circle*{0.2}}
	\put(6,7){\circle*{0.2}}
	\put(2,3.75){\circle*{0.2}}
	\put(3,0){\circle*{0.2}}
	\put(4,4){\circle*{0.2}}
	\put(8,2){\circle*{0.2}}
	\put(4,5){\circle*{0.2}}
	\put(1.5,7.2){$(\mcz_{12,3})$}
	\put(1.4,0.1){$A$}
	\put(6.1,3.3){$B$}
	\put(4.2,6.7){$C$}
	\put(4.1,2.7){$D$}
	\put(0.0,5.2){$E$}
	\put(4.1,1.4){$F$}
	\put(6.2,6.7){$G$}
	\put(1.3,3.3){$H$}
	\put(3.2,0.1){$I$}
	\put(3.8,4.2){$J$}
	\put(7.8,2.3){$K$}
	\put(3.8,5.2){$L$}
	\color{blue}
	\linethickness{0.33mm}
	\path(2,0)(6,3.25)(5,7)(4,3)(0,5)(4,2)(6,7)(2,3.75)(3,0)(4,4)(8,2)(4,5)(2,0)
\end{picture}
}

\vlskip

\centerline{%
\setlength{\unitlength}{8mm}
\begin{picture}(9,8)
	\put(0,2){\circle*{0.2}}
	\put(7.3,1.7){\circle*{0.2}}
	\put(8.5,6.7){\circle*{0.2}}
	\put(4.5,2.6){\circle*{0.2}}
	\put(1.3,0.0){\circle*{0.2}}
	\put(8,1.7){\circle*{0.2}}
	\put(1.1,7){\circle*{0.2}}
	\put(4.5,2.3){\circle*{0.2}}
	\put(9,0){\circle*{0.2}}
	\put(7.2,2.1){\circle*{0.2}}
	\put(1.9,2.1){\circle*{0.2}}
	\put(7.6,6.7){\circle*{0.2}}
	\put(8.5,0.4){\circle*{0.2}}
	\put(5.7,3){\circle*{0.2}}
	\put(4,6.5){$(\mcz_{14,3})$}
	\put(-0.3,1.3){$A$}
	\put(6.8,1.0){$B$}
	\put(8.8,6.6){$C$}
	\put(4.0,3.0){$D$}
	\put(0.8,0.3){$E$}
	\put(8.0,1.9){$F$}
	\put(0.5,6.2){$G$}
	\put(4.2,1.5){$H$}
	\put(9.3,-0.1){$I$}
	\put(6.8,2.3){$J$}
	\put(1.4,2.4){$K$}
	\put(6.8,6.7){$L$}
	\put(8.6,0.5){$M$}
	\put(5.8,3.1){$N$}
	\color{blue}
	\linethickness{0.33mm}
	\path(0,2)(7.3,1.7)(8.5,6.7)(4.5,2.6)(1.3,0.0)(8,1.7)(1.1,7)(4.5,2.3)(9,0)(7.2,2.1)(1.9,2.1)(7.6,6.7)(8.5,0.4)(5.7,3)(0,2)
\end{picture}
\kern4ex
\setlength{\unitlength}{5mm}
\begin{picture}(15,16)
	\put(1,0){\circle*{0.3}}
	\put(3,10){\circle*{0.3}}
	\put(13,4){\circle*{0.3}}
	\put(6,0){\circle*{0.3}}
	\put(8,5){\circle*{0.3}}
	\put(0,3){\circle*{0.3}}
	\put(5,6){\circle*{0.3}}
	\put(8,12){\circle*{0.3}}
	\put(9,5){\circle*{0.3}}
	\put(8,0){\circle*{0.3}}
	\put(4,5){\circle*{0.3}}
	\put(1,15){\circle*{0.3}}
	\put(12,5){\circle*{0.3}}
	\put(6,4){\circle*{0.3}}
	\put(15,3){\circle*{0.3}}
	\put(2,13){\circle*{0.3}}
	\put(12,10){$(\mcz_{16,3})$}
	\put(1.2,0.1){$A$}
	\put(3.0,10.1){$B$}
	\put(13.2,4.1){$C$}
	\put(4.9,0.1){$D$}
	\put(7.5,5.3){$E$}
	\put(-0.4,3.3){$F$}
	\put(4.3,6.3){$G$}
	\put(8.2,12.0){$H$}
	\put(9.3,4.9){$I$}
	\put(8.3,0.0){$J$}
	\put(3.1,4.0){$K$}
	\put(0.1,14.3){$L$}
	\put(12.1,5.2){$M$}
	\put(5.6,2.9){$N$}
	\put(14.1,2.0){$P$}
	\put(2.2,13.2){$Q$}
	\color{blue}
	\linethickness{0.33mm}
	\path(1,0)(3,10)(13,4)(6,0)(8,5)(0,3)(5,6)(8,12)(9,5)(8,0)(4,5)(1,15)(12,5)(6,4)(15,3)(2,13)(1,0)
\end{picture}
}

\vlskip

The next denizen of our zoo is relatively tame and very helpful. This polyline shows that type~\nck{14}{7} is feasible and therefore yields us an example of a type \nck{2k}{k} polyline for odd $k$ of form~$4p+3$ (for any odd $k$ of form $4p+1$ such a polyline can be easily constructed by applying procedure \kcb{12} to star~$\wstar(4p+1,2p)$). See page~$\pageref{use:nck_z7}$ for the use of this tricky configuration.

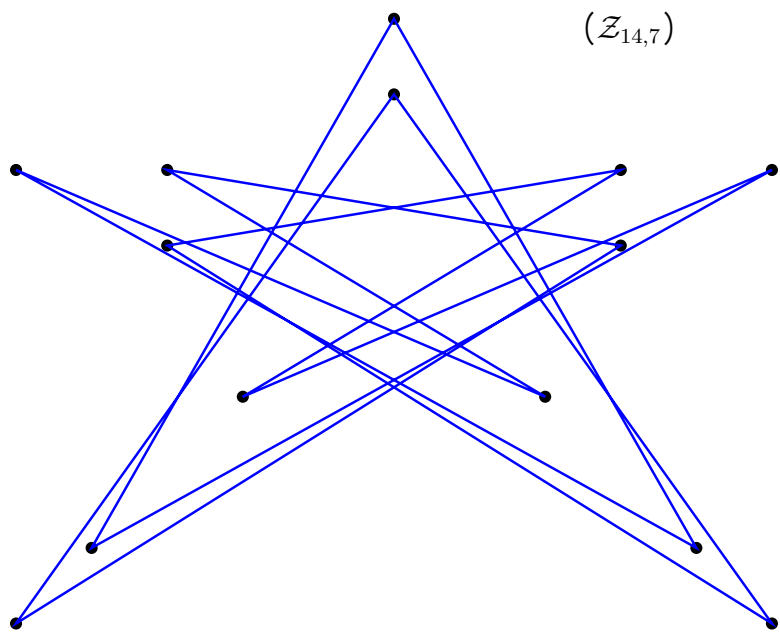
\vmskip

\hlabel{fig:nck_z7}

\centerline{\setlength{\unitlength}{10mm}
\begin{picture}(10,8.5)
  \put(10.00, 6.00){\circle*{0.16}}
  \put(1.00, 1.00){\circle*{0.16}}
  \put(5.00, 8.00){\circle*{0.16}}
  \put(9.00, 1.00){\circle*{0.16}}
  \put(0.00, 6.00){\circle*{0.16}}
  \put(7.00, 3.00){\circle*{0.16}}
  \put(2.00, 6.00){\circle*{0.16}}
  \put(8.00, 5.00){\circle*{0.16}}
  \put(0.00, 0.00){\circle*{0.16}}
  \put(5.00, 7.00){\circle*{0.16}}
  \put(10.00, 0.00){\circle*{0.16}}
  \put(2.00, 5.00){\circle*{0.16}}
  \put(8.00, 6.00){\circle*{0.16}}
  \put(3.00, 3.00){\circle*{0.16}}
  \put(7.5,7.75){$(\mcz_{14,7})$}
  \color{blue}
  \linethickness{0.33mm}
  \path(10.00, 6.00)(1.00, 1.00)(5.00, 8.00)(9.00, 1.00)(0.00, 6.00)(7.00, 3.00)(2.00, 6.00)(8.00, 5.00)(0.00, 0.00)(5.00, 7.00)(10.00, 0.00)(2.00, 5.00)(8.00, 6.00)(3.00, 3.00)(10.00, 6.00)
\end{picture}
}

\vmskip

And here is yet another polyline, not too exotic, but not very simple either. Its self-intersection index is equal to $5$ and we used it as a part in our computation of $\mcc_5$ on page~$\pageref{use:nck_z5}$.

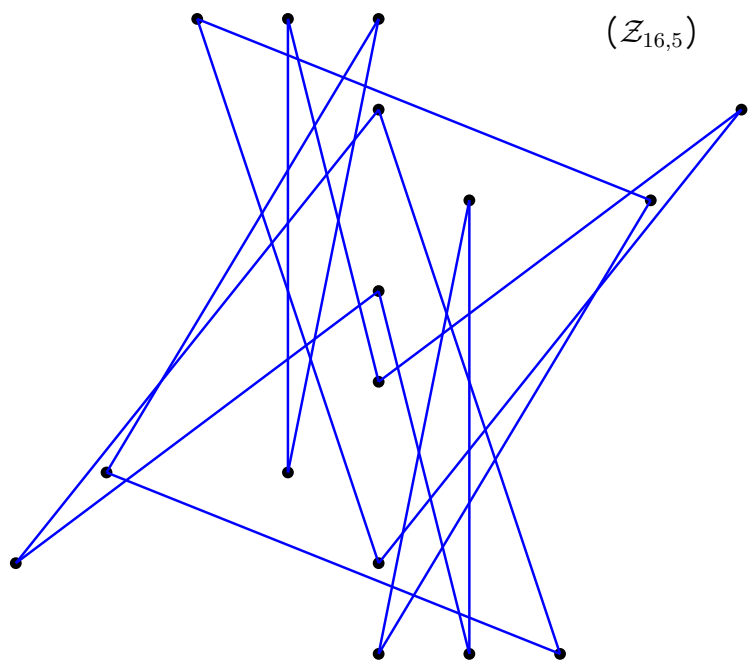
\vmskip

\hlabel{fig:nck_z5}
\centerline{%
\setlength{\unitlength}{12mm}
\begin{picture}(9,8)
  \put(0.00, 1.00){\circle*{0.13}}
  \put(4.00, 4.00){\circle*{0.13}}
  \put(5.00, 0.00){\circle*{0.13}}
  \put(5.00, 5.00){\circle*{0.13}}
  \put(4.00, 0.00){\circle*{0.13}}
  \put(7.00, 5.00){\circle*{0.13}}
  \put(2.00, 7.00){\circle*{0.13}}
  \put(4.00, 1.00){\circle*{0.13}}
  \put(8.00, 6.00){\circle*{0.13}}
  \put(4.00, 3.00){\circle*{0.13}}
  \put(3.00, 7.00){\circle*{0.13}}
  \put(3.00, 2.00){\circle*{0.13}}
  \put(4.00, 7.00){\circle*{0.13}}
  \put(1.00, 2.00){\circle*{0.13}}
  \put(6.00, 0.00){\circle*{0.13}}
  \put(4.00, 6.00){\circle*{0.13}}
  \put(6.5,6.75){$(\mcz_{16,5})$}
  \color{blue}
  \linethickness{0.33mm}
  \path(0.00, 1.00)(4.00, 4.00)
  \path(4.00, 4.00)(5.00, 0.00)
  \path(5.00, 0.00)(5.00, 5.00)
  \path(5.00, 5.00)(4.00, 0.00)
  \path(4.00, 0.00)(7.00, 5.00)
  \path(7.00, 5.00)(2.00, 7.00)
  \path(2.00, 7.00)(4.00, 1.00)
  \path(4.00, 1.00)(8.00, 6.00)
  \path(8.00, 6.00)(4.00, 3.00)
  \path(4.00, 3.00)(3.00, 7.00)
  \path(3.00, 7.00)(3.00, 2.00)
  \path(3.00, 2.00)(4.00, 7.00)
  \path(4.00, 7.00)(1.00, 2.00)
  \path(1.00, 2.00)(6.00, 0.00)
  \path(6.00, 0.00)(4.00, 6.00)
  \path(4.00, 6.00)(0.00, 1.00)
\end{picture}
}

\vmskip

The next cage houses the family of a few very convoluted polylines which we needed for computing $\mcb_{42}$ on page~$\pageref{use:nck_42z}$. They all belong to types~\nck{42}{k}. Here, for the sake of sanity, we will only present one of them---namely, polyline~$\mcz_{42,19}$.

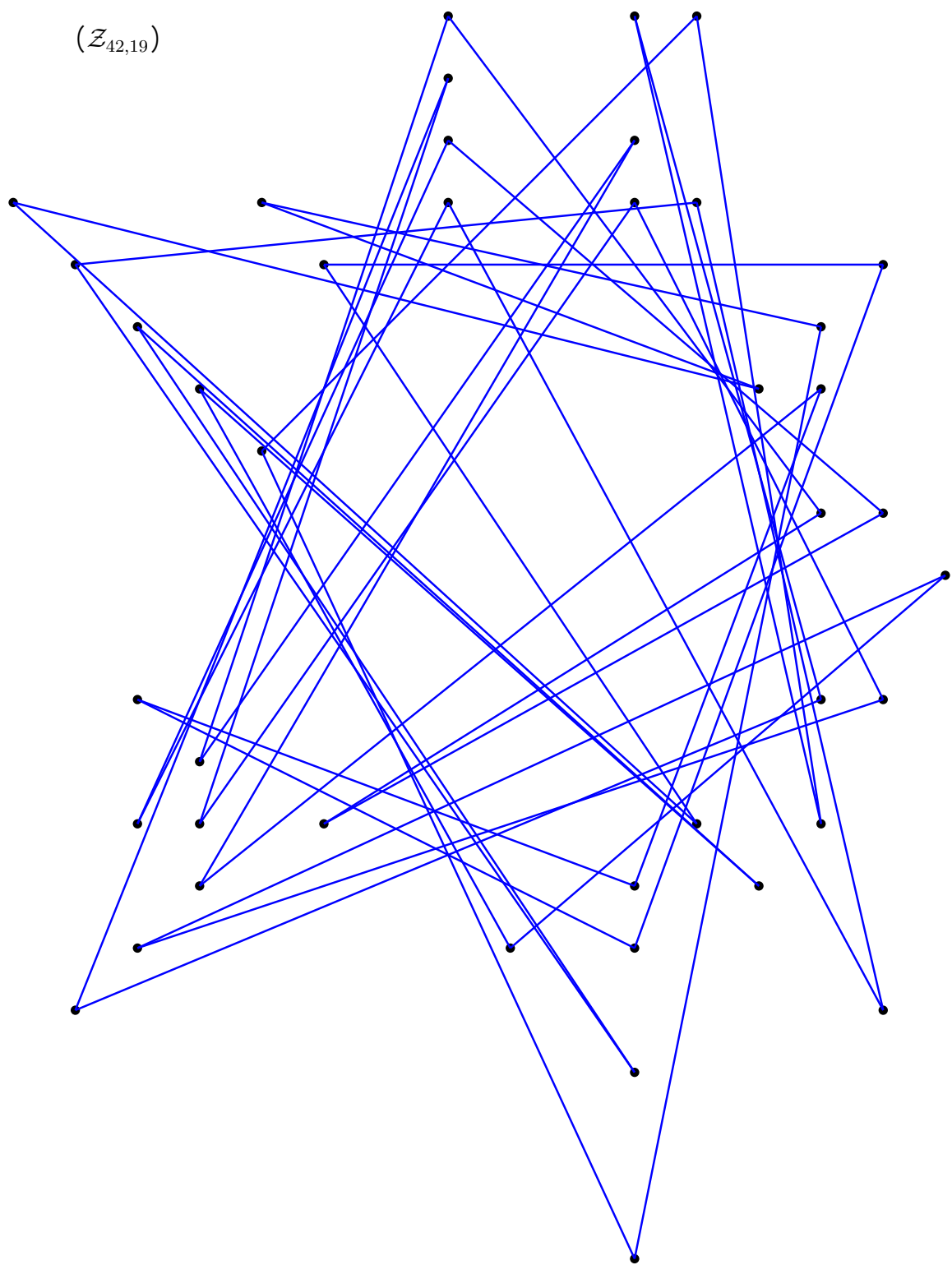
\vmskip

\centerline{%
\hlabel{fig:nck_42z}
\setlength{\unitlength}{10mm}
\begin{picture}(15,21)
  \put(14.0, 12.0){\circle*{0.15}}
  \put(7.0, 18.0){\circle*{0.15}}
  \put(2.0, 7.0){\circle*{0.15}}
  \put(7.0, 17.0){\circle*{0.15}}
  \put(14.0, 4.0){\circle*{0.15}}
  \put(11.0, 17.0){\circle*{0.15}}
  \put(1.0, 16.0){\circle*{0.15}}
  \put(10.0, 3.0){\circle*{0.15}}
  \put(2.0, 15.0){\circle*{0.15}}
  \put(12.0, 6.0){\circle*{0.15}}
  \put(3.0, 14.0){\circle*{0.15}}
  \put(8.0, 5.0){\circle*{0.15}}
  \put(15.0, 11.0){\circle*{0.15}}
  \put(2.0, 5.0){\circle*{0.15}}
  \put(14.0, 9.0){\circle*{0.15}}
  \put(10.0, 17.0){\circle*{0.15}}
  \put(3.0, 7.0){\circle*{0.15}}
  \put(7.0, 19.0){\circle*{0.15}}
  \put(1.0, 4.0){\circle*{0.15}}
  \put(13.0, 9.0){\circle*{0.15}}
  \put(10.0, 20.0){\circle*{0.15}}
  \put(13.0, 7.0){\circle*{0.15}}
  \put(11.0, 20.0){\circle*{0.15}}
  \put(4.0, 13.0){\circle*{0.15}}
  \put(10.0, 0.0){\circle*{0.15}}
  \put(13.0, 15.0){\circle*{0.15}}
  \put(4.0, 17.0){\circle*{0.15}}
  \put(12.0, 14.0){\circle*{0.15}}
  \put(0.0, 17.0){\circle*{0.15}}
  \put(11.0, 7.0){\circle*{0.15}}
  \put(5.0, 16.0){\circle*{0.15}}
  \put(14.0, 16.0){\circle*{0.15}}
  \put(10.0, 5.0){\circle*{0.15}}
  \put(2.0, 9.0){\circle*{0.15}}
  \put(10.0, 6.0){\circle*{0.15}}
  \put(13.0, 14.0){\circle*{0.15}}
  \put(3.0, 6.0){\circle*{0.15}}
  \put(10.0, 18.0){\circle*{0.15}}
  \put(3.0, 8.0){\circle*{0.15}}
  \put(7.0, 20.0){\circle*{0.15}}
  \put(13.0, 12.0){\circle*{0.15}}
  \put(5.0, 7.0){\circle*{0.15}}
  \put(1,19.5){$(\mcz_{42,19})$}
  \color{blue}
  \linethickness{0.33mm}
  \path(14.0, 12.0)(7.0, 18.0)
  \path(7.0, 18.0)(2.0, 7.0)
  \path(2.0, 7.0)(7.0, 17.0)
  \path(7.0, 17.0)(14.0, 4.0)
  \path(14.0, 4.0)(11.0, 17.0)
  \path(11.0, 17.0)(1.0, 16.0)
  \path(1.0, 16.0)(10.0, 3.0)
  \path(10.0, 3.0)(2.0, 15.0)
  \path(2.0, 15.0)(12.0, 6.0)
  \path(12.0, 6.0)(3.0, 14.0)
  \path(3.0, 14.0)(8.0, 5.0)
  \path(8.0, 5.0)(15.0, 11.0)
  \path(15.0, 11.0)(2.0, 5.0)
  \path(2.0, 5.0)(14.0, 9.0)
  \path(14.0, 9.0)(10.0, 17.0)
  \path(10.0, 17.0)(3.0, 7.0)
  \path(3.0, 7.0)(7.0, 19.0)
  \path(7.0, 19.0)(1.0, 4.0)
  \path(1.0, 4.0)(13.0, 9.0)
  \path(13.0, 9.0)(10.0, 20.0)
  \path(10.0, 20.0)(13.0, 7.0)
  \path(13.0, 7.0)(11.0, 20.0)
  \path(11.0, 20.0)(4.0, 13.0)
  \path(4.0, 13.0)(10.0, 0.0)
  \path(10.0, 0.0)(13.0, 15.0)
  \path(13.0, 15.0)(4.0, 17.0)
  \path(4.0, 17.0)(12.0, 14.0)
  \path(12.0, 14.0)(0.0, 17.0)
  \path(0.0, 17.0)(11.0, 7.0)
  \path(11.0, 7.0)(5.0, 16.0)
  \path(5.0, 16.0)(14.0, 16.0)
  \path(14.0, 16.0)(10.0, 5.0)
  \path(10.0, 5.0)(2.0, 9.0)
  \path(2.0, 9.0)(10.0, 6.0)
  \path(10.0, 6.0)(13.0, 14.0)
  \path(13.0, 14.0)(3.0, 6.0)
  \path(3.0, 6.0)(10.0, 18.0)
  \path(10.0, 18.0)(3.0, 8.0)
  \path(3.0, 8.0)(7.0, 20.0)
  \path(7.0, 20.0)(13.0, 12.0)
  \path(13.0, 12.0)(5.0, 7.0)
  \path(5.0, 7.0)(14.0, 12.0)
\end{picture}%
}

\vmskip

If our meticulous readers are interested in getting the exact data underlying the polylines $\mcz_{42,k}$ for the values $k=19, 23, 26, 31, 33, 35$, that information can be easily obtained by contacting either the author or the curator of this wonderful zoo, Konstantin A.~Knop.

\end{document}